\documentclass[11pt]{article}
\usepackage[numbers,sort&compress]{natbib}
\usepackage{enumerate}
\usepackage{amscd}
\usepackage{amsmath}
\usepackage{latexsym}
\usepackage{amsfonts}
\usepackage{setspace}
\usepackage{amssymb}
\usepackage{amsthm}
\usepackage{verbatim}
\usepackage{mathrsfs}
\usepackage{enumerate}
\usepackage[hypertexnames=false]{hyperref}

\theoremstyle{plain}
\theoremstyle{plain}\newtheorem{theorem}{Theorem}[section]
\theoremstyle{plain}\newtheorem{lemma}[theorem]{Lemma}
\theoremstyle{plain}
\theoremstyle{plain}\newtheorem{prop}[theorem]{Proposition}
\theoremstyle{definition}\newtheorem{remark}{Remark}[section]
\usepackage{xcolor}

\newcommand{\mr}{\mathbb{R}}
\newcommand{\mt}{\mathbb{T}}

\newcommand{\ben}{\begin{enumerate}}
	\newcommand{\een}{\end{enumerate}}

\newcommand{\bv}{\mathbf{v} }
\newcommand{\bV}{\mathbf{V} }
\newcommand{\bH}{\mathbf{H} }
\newcommand{\bb}{\mathbf{b} }

\newcommand{\bj}{\mathbf{j} }
\newcommand{\0}{\mathbf{0} }
\newcommand{\bx}{\mathbf{x} }

\newcommand{\Rmnum}[1]{\expandafter\@slowromancap\romannumeral #1@}

\allowdisplaybreaks

\numberwithin{equation}{section}
\begin{document}

%%%%%%%%%%%%%%%%%%%%%%%%%%%%%%%%%%%%%%%%%%%%%%%%%%%%%%%%%%%%%%%%%%%%%%%%%%%%%%%%%%%%%%%%%%%%%%%%%%%%
	\title{Global stability and asymptotic behavior for incompressible ideal MHD equations with  velocity damping term}
\author{Hui Fang\,\,Pingping Gui,\,\,Yanping Zhou}
\date{}
\maketitle

\begin{abstract}
In this article, we study the stability and large time behavior for an multi-dimensional incompressible magnetohydrodynamical system with a velocity damping term, for small perturbations near a steady-state of magnetic field fulfilling the Diophantine condition. Our results  mathematically characterize the background magnetic field exerts the stabilizing effect, and bridge the gap left by previous work with respect to the asymptotic  behavior in time. Our proof  approach mainly relies on the Fourier analysis and energy estimates. In addition, we provide a versatile analytical framework applicable to many other partially dissipative fluid models.

\end{abstract}
\noindent {\bf MSC(2020):}\quad 35Q35, 35L45, 35B40.
    \vskip 0.02cm
	\noindent {\bf Keywords:} Ideal MHD equations, damping, stability, decay estimates.

\section{Introduction}
The magnetohydrodynamics (MHD) equations characterize the motion of electrically conducting flows such as electrolytes, liquid metals and plasmas, and play a foundational role in geophysical and astrophysical fluids (see, for example, \cite{Biskamp1993,PriestForbes2000,Davidson2001,Duvaut1972}).
When resistive effects are extremely weak, magnetic diffusion becomes negligible, leading to the so-called non-resistive MHD system, which takes the form
\begin{align}\label{mhd1}
\left\{
\begin{array}{l}
\partial_t \bv + (-\Delta)^{\alpha} \bv + \bv \cdot \nabla \bv + \nabla p = \bb \cdot \nabla \bb, \\[1mm]
\partial_t \bb + \bv \cdot \nabla \bb = \bb \cdot \nabla \bv, \\[1mm]
\nabla \cdot \bv = \nabla \cdot \bb = 0, \\[1mm]
\bv(\bx,0) = \bv_0(\bx), \quad \bb(\bx,0) = \bb_0(\bx),
\end{array}
\right.
\end{align}
where \(\bv(\bx,t)\), \(\bb(\bx,t)\), and \(p(\bx,t)\) denote the velocity, magnetic field, and pressure, respectively. The parameter \( \alpha \ge 0 \) characterizes the strength of velocity dissipation through the fractional Laplacian \( (-\Delta)^{\alpha} \); in particular, \( \alpha = 0 \) corresponds to a damping effect modeling the frictional relaxation of the flow. System \eqref{mhd1} describes plasmas that are strongly collisional or have extremely small resistivity due to collisions. It is relevant when the characteristic spatial scales are much larger than the ion skin depth and Larmor radius, and the temporal scales are much longer than the ion gyration period, so that magnetic diffusion can be neglected (see, e.g., \cite{Cabannes1970, Cowling1976, LandauLifshitz1984,Landau1984}).

Mathematically, system \eqref{mhd1} shares key structural features with the Euler and Navier-Stokes equations, while the magnetic coupling and the absence of resistivity introduce richer dynamics and substantial analytical difficulties.
Two fundamental issues concerning \eqref{mhd1} have attracted central attention.
The first is the \emph{global well-posedness} problem: Chemin et al. \cite{Chemin2016} got the local well-posedness in critical Besov spaces, whereas the global existence of classical solutions remains open even in two dimensions. The second is the \emph{stability} problem for perturbations around a background magnetic field.
A background field, say
\[
(\bv,\bb)=({\mathbf{0}},\widetilde{\bb})
\]
defines a family of steady states. Writing \(\bb\) for the perturbation \(\bb-\widetilde{\bb}\), one obtains the perturbed MHD system with two extra terms
\begin{align}\label{mhd11}
\left\{
\begin{aligned}
&\partial_t \bv + (-\Delta)^{\alpha} \bv + \bv\cdot\nabla \bv + \nabla p = \widetilde{\bb}\cdot\nabla \bb + \bb\cdot\nabla \bb, \\
&\partial_t \bb + u\cdot\nabla \bb = \widetilde{\bb}\cdot\nabla \bv + \bb\cdot\nabla \bv, \\
&\nabla\cdot \bv = \nabla\cdot \bb = 0, \\
&\bv(\bx,0)=\bv_0(\bx),\quad \bb(\bx,0)=\bb_0(\bx).
\end{aligned}
\right.
\end{align}
The stability problem for \eqref{mhd11} is motivated by the physical observation-
confirmed by both experiments and numerical simulations-that magnetic fields can stabilize electrically conducting fluids (see, e.g., \cite{HAlfvn1942,Alemany-Frisch1979,Davidson1995,Davidson1997,Alexakis2011,Gallet-Mordant2009,Gallet-Doering2015,Califano19999,Majda1984}).
Recent mathematical works have mainly focused on two classes of background fields:
(i) \emph{strong directional fields}, for example, $\overline \bb = \mathbf{e}_n$ along the last coordinate direction in $\mathbb{R}^n$, and
(ii) \emph{Diophantine fields} satisfying the non-resonance condition
\begin{align}\label{Diophantine}
|\widetilde \bb\cdot \bj|\ge \frac{c}{|\bj|^r},\qquad \bj\in\mathbb{Z}^n\setminus\{\0\},
\end{align}
for some $r>n-1$ and $c>0$. The Diophantine condition holds for almost all $\widetilde \bb\in\mathbb{R}^n$, except when its components are rational or one of them vanishes~\cite{chen-2022-3dmhd-Diophant}. Note that a strong directional field does not satisfy the Diophantine condition.

For the viscous case $\alpha = 1$ in \eqref{mhd11}, the stability problem under a strong background magnetic field $\overline \bb$ has a long history.
The analysis of well-posedness was initiated by Lin and Zhang~\cite{lin-2014-GlobalSmallSolutions} for a related three-dimensional model (see also~\cite{Lin-Zhang2015simplified-proof}), and was later extended by Lin et al.~\cite{Lin-Xu-Zhang2015JDE} in two dimensions and by Xu and Zhang~\cite{xu-2015-GlobalSmallSolutions} in three dimensions.
More results on the stability and long-time behavior under strong magnetic fields can be found in~\cite{ren2014global,zhangTJDE2016,Abidi-Zhang2017,Pan-Zhou-Zhu2018,Deng2018zhang,jiang2021}.
The earliest rigorous progress in this direction can be traced back to Bardos et al. \cite{Bardos-Sulem1988}, who proved the global well-posedness of the ideal incompressible MHD system for small perturbations around a nontrivial equilibrium, showing that a sufficiently strong magnetic field can suppress nonlinear interactions and prevent the formation of large gradients~\cite{Frisch1983,Kraichnan1965}.

For background fields fulfilling the Diophantine condition~\eqref{Diophantine},
Chen et al.~\cite{chen-2022-3dmhd-Diophant} first established global asymptotic stability in the three-dimensional periodic domain $\mathbb{T}^3$, proving convergence in $H^{4r+7}(\mathbb{T}^3)$ for $r>n-1$.
Zhai~\cite{zhai-2023-2dmhdstability-Diophant} later refined this result in the case of two dimensions $\mathbb{T}^2$, reducing the regularity to $H^{(3+2\beta)r+5+(\gamma+2\beta)}(\mathbb{T}^2)$ for arbitrary $\beta,\gamma>0$, while Xie et al.~\cite{XieJiu-CVPDE2024} lowered the threshold to $H^{(3r+3)^+}(\mathbb{T}^n)$ for both two- and three-dimensional settings.
In our recent work~\cite{Bie}, we further reduced the regularity requirement to \(H^{(3 + 2r + \tfrac{n}{2})^+}(\mathbb{T}^n)\).

 For the inviscid case $\alpha = 0$ in \eqref{mhd11} with a strong background magnetic field,
Wu et al.~\cite{wu-wu-xu-2015-GlobalSmallSolutionMHD} first analyzed the two dimensional system in the whole space, proving global stability of small perturbations together with explicit long-time decay rates in various Sobolev norms.
Jo et al.~\cite{Lee} later improved these results by relaxing certain regularity assumptions and extending the decay analysis.
Du et al.~\cite{Du2019} established exponential stability in a strip domain $\mathbb{R}\times[0,1]$, while Jiang et al.~\cite{Jiang2022} proved global existence and exponential decay of classical solutions in a horizontally periodic strip domain $\mathbb{T}^2\times[0,1]$.
For further results on inviscid systems with velocity damping, see~\cite{Tan2013,Wang2001,Pan2009,Sideris2003,Huang2003}.

More recently, Zhao and Zhai~\cite{zhao-zhai-2021-GlobalSmallSolutions3DMHD} investigated the three-dimensional inviscid system corresponding to $\alpha = 0$  in \eqref{mhd11}, namely,
\begin{align}\label{equation}
\left\{
\begin{array}{l}
\partial_t \bv + \bv + \bv \cdot \nabla \bv + \nabla p = \widetilde{\bb}  \cdot \nabla \bb + \bb \cdot \nabla \bb, \\[1mm]
\partial_t \bb + \bv \cdot \nabla \bb = \widetilde{\bb}  \cdot \nabla \bv + \bb \cdot \nabla \bv, \\[1mm]
\nabla\cdot \bv = \nabla \cdot \bb = 0, \\[1mm]
\bv(\bx,0)=\bv_0(\bx),\quad \bb(\bx,0)=\bb_0(\bx),
\end{array}
\right.
\end{align}
and established the global stability and asymptotic decay of small perturbations on the periodic domain $\mathbb{T}^3$. Inspired by~\cite{Elgindi2010,Jiang-Kim2023CVPDE}, in the present paper,  we still study the model~\eqref{equation}, whereas  on the $n$-dimensional torus $\mathbb{T}^n$ ($n \ge 2$), and further improve the results  with respect to the global stability and asymptotic behavior in time.

To capture the essential decay structure and clarify the underlying dynamics, we first study the corresponding linearized system:
\begin{align}\label{lineartheorem}
\left\{ \begin{array}{l} \partial_t \bV - \Delta \bV = \widetilde{\bb}\cdot\nabla \bH, \\[0.4ex] \partial_t \bH = \widetilde{\bb}\cdot\nabla \bV, \\[0.4ex] \nabla\cdot \bV = \nabla\cdot \bH = 0, \\[0.4ex] \bV(\bx,0)=\bV_0(\bx), \quad \bH(x,0)=\bH_0(\bx),
\end{array} \right.
\end{align}
whose decay properties provide the foundation for the nonlinear analysis. The following theorem describes the decay behavior of the linearized system~\eqref{lineartheorem}.
\begin{theorem}\label{thm1}
Let \( n \ge 2 \) and \( r > n - 1 \), and suppose that the background magnetic field \( \widetilde{\bb} \) adheres to the Diophantine condition.
Assume that the initial data \( (\bV_0, \bH_0) \in H^m(\mathbb{T}^n) \) for some integer \( m \ge 1 \) satisfy the mean-free conditions
\begin{align}\label{mean-zero}
\int_{\mathbb{T}^n} \bV_0\, {d}\bx = \0,
\qquad
\int_{\mathbb{T}^n} \bH_0\, {d}\bx = \0.
\end{align}
Then the corresponding solution \( (\bV,\bH) \) to \eqref{lineartheorem} fulfills
\begin{subequations}\label{linearm}
\begin{align}
\|\bV(t)\|_{H^s(\mathbb{T}^n)}
&\le
\begin{cases}
C(1+t)^{-\frac{1}{2}-\frac{m-s}{2r}} \, \|(\bV_0,\bH_0)\|_{H^{m}(\mathbb{T}^n)}, & 0 \le s \le m-1, \\[0.8ex]
C\,\|(\bV_0,\bH_0)\|_{H^{m}(\mathbb{T}^n)}, & s = m,
\end{cases} \\[1ex]
\|\bH(t)\|_{H^s(\mathbb{T}^n)}
&\le C(1+t)^{-\frac{m-s}{2r}} \, \|(\bV_0,\bH_0)\|_{H^{m}(\mathbb{T}^n)}, \qquad 0 \le s \le m,
\end{align}
\end{subequations}
where \(C>0\) denotes a constant independent of \(t\).
\end{theorem}

\begin{remark}
 The work~\cite{Bie} investigated the viscous counterpart of~\eqref{lineartheorem}, where the velocity damping term \(\bV\) is replaced by the Laplacian dissipation \(-\Delta \bV\). It was shown that for any \(s \in [0,m]\),
\begin{subequations}\label{alpha1}
\begin{align}
\|\bV(t)\|_{H^s} &\le C(1+t)^{-\left(\frac{1}{2}+\frac{m-s+1}{2(1+r)}\right)} \|(\bV_0,\bH_0)\|_{H^m},\\[0.6ex]
\|\bH(t)\|_{H^s} &\le C(1+t)^{-\frac{m-s}{2(1+r)}} \|(\bV_0,\bH_0)\|_{H^m},
\end{align}
\end{subequations}
which, compared with~\eqref{linearm}, shows that velocity damping yields faster decay than viscous diffusion for all \(s \in [0,m-1]\). This difference arises because direct damping acts uniformly on all frequency modes, while viscous diffusion becomes less effective at high frequencies due to coupling with the magnetic field.  A comparable mechanism has also been observed in the Boussinesq system~\cite{Jiang-Kim2023CVPDE}.
\end{remark}

Our next theorem addresses the nonlinear stability for MHD equations \eqref{equation}.
\begin{theorem}\label{thm}
Let \(m \in \mathbb{N}\) satisfy
\begin{align}\label{m1}
m >
\begin{cases}
5, & n = 2,\; 1 < r \le \tfrac{3}{2}, \\[1mm]
2 + 2r, & n = 2,\; r > \tfrac{3}{2}, \\[1mm]
1 + 2r + \tfrac{n}{2}, & n \ge 3,\; r>n-1.
\end{cases}
\end{align}
Suppose that \(\widetilde{\bb}\) fulfills the Diophantine condition, and that the initial data \((\bv_0, \bb_0) \in H^m(\mathbb{T}^n)\) satisfy
\[
\mathrm{div}\, \bv_0 = \mathrm{div}\, \bb_0 = 0, \quad
\int_{\mathbb{T}^n} \bv_0\, {d}\bx = \int_{\mathbb{T}^n} \bb_0\, {d}\bx = \0,
\]
and
\begin{align}\label{smallcondition}
	\|\bv_0\|_{H^m(\mt^n)}+\|\bb_0\|_{H^m(\mt^n)}\leq \varepsilon,
\end{align}
for a small enough constant \(\varepsilon > 0\).
Then the incompressible MHD system \eqref{equation} admits a unique global classical solution \((\bv, \bb)\) with
\begin{align*}
&\bv \in C\big([0, \infty); H^m(\mathbb{T}^n)\big) \cap L^2\big([0, \infty); H^{m}(\mathbb{T}^n)\big),\\
 &\bb \in C\big([0, \infty); H^m(\mathbb{T}^n)\big)\cap L^2\big([0, \infty); H^{m-r-1}(\mathbb{T}^n)\big),
\end{align*}
satisfying the energy bound
\begin{align}\label{eq:1.2}
\sup_{t \in [0, \infty)} \lVert (\bv, \bb)(t) \rVert_{H^m}^2
+ \int_0^\infty \lVert \bv(t) \rVert_{H^m}^2 \, {d}t
+ \int_0^\infty \lVert \bb(t) \rVert_{H^{m-1-r}}^2 \, {d}t
\leq C \lVert (\bv_0, \bb_0) \rVert_{H^m}^2,
\end{align}
and the decay estimate
\begin{align}\label{finaldecay}
\lVert (\bv, \bb)(t) \rVert_{H^s(\mt^n)} \leq C(1 + t)^{-\frac{m - s}{2(1 + r)}} \quad \text{for any } s \in [0, m].
\end{align}
\end{theorem}

\begin{remark}
Zhao and Zhai~\cite{zhao-zhai-2021-GlobalSmallSolutions3DMHD} established the global stability of system \eqref{equation} for small initial data in $H^{m}(\mathbb{T}^3)$ with $m \ge 4r + 7$ and $r > 2$, and proved the decay estimate
\begin{equation}\label{ZZ-decay}
 \lVert (\bv, \bb)(t) \rVert_{H^s(\mathbb{T}^3)}
 \leq C (1 + t)^{-\frac{3(m - s)}{2(m - r - 4)}} ,
 \qquad s \in [r + 4, m).
\end{equation}
Theorem~\ref{thm} extends this result to $n$-dimensional periodic domains ($n \ge 2$) and strengthens~\eqref{ZZ-decay} in the following aspects:

\smallskip
\noindent
(1) \textit{Lower regularity.}
The required initial regularity is reduced from $H^{4r+7}(\mathbb{T}^3)$ to $H^{(1 + 2r + \frac{3}{2})^+}(\mathbb{T}^3)$.

\smallskip
\noindent
(2) \textit{Wider validity range.}
The decay estimate~\eqref{finaldecay} applies to all Sobolev indices $s \in [0, m]$, extending the range $s \ge r + 4$ in~\eqref{ZZ-decay} and yielding a complete low-to-high regularity decay profile.

\smallskip
\noindent
(3) \textit{Sharper decay exponent.}
\begin{itemize}
    \item When $m = 4r + 7$, both results give the same rate:
    \[
    \lVert (\bv, \bb)(t) \rVert_{H^s(\mathbb{T}^3)} \leq C(1 + t)^{-\frac{4r + 7 - s}{2(1 + r)}},
    \qquad s \in [r + 4, 4r + 7).
    \]

    \item When $m > 4r + 7$, our estimate yields strictly faster decay for any fixed $s$. It is worth mentioning  that the temporal decay rates increase with $m$, i.e., the more regular the initial data, the faster the temporal decay. In the limit $m \to +\infty$ (arbitrarily smooth initial data),
    \[
    \lim_{m \to +\infty}\frac{3(m - s)}{2(m - r - 4)} = \tfrac{3}{2}, \qquad
    \lim_{m \to +\infty}\frac{m - s}{2(1 + r)} = +\infty,
    \]
    showing that the algebraic decay exponent in~\eqref{finaldecay} can become arbitrarily large in magnitude (i.e., the decay becomes arbitrarily fast), while that in~\eqref{ZZ-decay} saturates at $-\tfrac{3}{2}$. This highlights the fundamentally stronger algebraic decay achieved in Theorem~\ref{thm}.
\end{itemize}
\end{remark}

\begin{remark}
A notable property of the velocity-damped MHD system is uncovered: the $L^2$-norm decay rate of the solution is higher when the initial value has a higher degree of regularity, whereas the $H^m$-norm decay rate remains unaffected. This behavior sharply contrasts with that of classical parabolic system, where higher  norms generally decay faster.
\end{remark}
\begin{remark}
The method  developed in this paper is dimension-independent, giving decay estimates in Sobolev spaces of considerably lower regularity. Moreover, it is capable of  being readily extended to other  PDEs with partially dissipative on  the space $\mathbb{T}^n$.
\end{remark}

To conclude, we briefly discuss the motivations for studying the stability and decay of the ideal MHD system with velocity damping from the following two aspects.

First, the velocity-damped model captures the essential coupling between velocity and magnetic perturbations while avoiding additional dispersive or resistive effects, providing a natural setting to explore the stabilizing influence of a constant background magnetic field. This effect is analogous to the inviscid damping phenomenon for the Euler equations near shear flows: although the \( \bb \)-equation in \eqref{mhd1} contains no diffusion, perturbations around \( \widetilde{\bb} \) induce phase mixing that suppresses nonlinear growth and yields decay at the linearized level. As shown in \eqref{equation}, the linear structure involving \( \widetilde{\bb}\cdot\nabla \bb \) and \( \widetilde{\bb}\cdot\nabla \bv \) naturally arises from the perturbative formulation of the system. This stabilization mechanism also resonates with the geometric perspective of Choffrut and \v{S}ver\'{a}~\cite{ChoffrutSverak}, which reflects the diversity of the properties of nearby steady states for the two-dimensional Euler flows; see also~\cite{Drivas}.

Second, system~\eqref{mhd1} is closely related to magnetic relaxation, a phenomenon first proposed by Arnol'd~\cite{Arnold1974} in 1974 and further developed by Moffatt~\cite{Moffatt1985,Moffatt2021}. The magnetic relaxation conjecture suggests that the magnetic field \( \bb \) asymptotically converges to a stationary Euler flow (or magnetostatic equilibrium), while the velocity decays due to kinetic dissipation.  Although the case \(\alpha = 1\) in \eqref{mhd1} was emphasized in~\cite{Moffatt1985}, Moffatt also anticipated that other dissipative mechanisms-including the velocity damping case \(\alpha = 0\)-could equally drive the relaxation process. In particular, the case \(\alpha = 0\) eliminates dependence on diffusive effects, making it a natural setting for investigating magnetic relaxation. Recently, Dai, Lai, and Zhang~\cite{Dai-Lai-Zhai2025} established global stability and large-time decay for the magnetic relaxation equations (MRE) on periodic domains \(\mathbb{T}^n\) \((n = 2,3)\) under the Diophantine condition.
The present model \eqref{mhd1} with \(\alpha = 0\) is closely connected to the MRE with \(\gamma = 0\); see also~\cite{McCormick2014C,Beekie2022} for related studies.

The organization  of this work is  as follows. In Section~\ref{2}, we present some preliminary lemmas and tools, including Fourier multiplier estimates and properties related to the Diophantine condition. In Section~\ref{sec:kernel}, taking advantage of spectral analysis,  we give an integral representation of the solution and obtain detailed kernel estimates through splitting  the frequency space into sub-domains. In Section~\ref{2/3}, we  prove the linear stability, while the nonlinear stability is shown in the final Section~\ref{s:6}.

\section{Preliminaries}\label{2}
In this section,  we provide some preliminaries, including the properties of spatial averages of the solution $(\bv,\bb)$, Poincar\'{e}-type inequality, and  a commutator estimate.
\begin{lemma}\label{lem:mean}
Let $(\bv,\bb)$ be a smooth solution to \eqref{equation} fulfilling
\[
\operatorname{div} \bv_0 = \operatorname{div} \bb_0 = 0, \quad \int_{\mathbb{T}^n} \bv_0\, {d}\bx = \int_{\mathbb{T}^n} \bb_0\, {d}\bx = \0.
\]
Then, for any \( t \geq 0 \), one has
\begin{equation}\label{meanzero2}
\int_{\mathbb{T}^n} \bv(\bx,t)\, {d}\bx =\int_{\mathbb{T}^n} \bb(\bx,t)\, {d}\bx=\0.
\end{equation}
\begin{proof}
Integrating the equation in \eqref{equation} on \( \mathbb{T}^n \), we deduce
\[
\frac{d}{dt} \int_{\mathbb{T}^n} \bv \,d\bx - \int_{\mathbb{T}^n} \bv \,d\bx + \int_{\mathbb{T}^n} \bv \cdot \nabla \bv \,d\bx + \int_{\mathbb{T}^n} \nabla p\,d\bx = \int_{\mathbb{T}^n} \widetilde{\bb} \cdot \nabla \bb\,d\bx + \int_{\mathbb{T}^n} \bb \cdot \nabla \bb\,d\bx,
\]
and
\begin{align*}
	\frac{d}{d t}\int_{\mathbb{T}^n}\bb\,d \bx+\int_{\mathbb{T}^n}\bv  \cdot \nabla \bb\,d \bx=\int_{\mathbb{T}^n}\widetilde{\bb} \cdot \nabla \bv \,d \bx+\int_{\mathbb{T}^n}\bb \cdot \nabla \bv \,d \bx.
\end{align*}
By $\nabla\cdot\bv=\nabla\cdot\bb=0$ and the periodic boundary condition, we infer by integration by parts that
\[
\int_{\mathbb{T}^n} \bv \cdot \nabla \bv  \, d\bx = \int_{\mathbb{T}^n} \nabla p \, d\bx = \int_{\mathbb{T}^n} \widetilde{\bb} \cdot \nabla \bb \, d\bx = \int_{\mathbb{T}^n} \bb \cdot \nabla \bb \, d\bx = \0,
\]
and
\begin{align*}
\int_{\mathbb{T}^n}\bv  \cdot \nabla \bb\,d\bx=\int_{\mathbb{T}^n}\widetilde{\bb} \cdot \nabla \bv \,d \bx=\int_{\mathbb{T}^n}\bb \cdot \nabla \bv \,d \bx=\0,
\end{align*}
which implies
\begin{align*}
	\frac{d}{d t}\int_{\mathbb{T}^n}\bv \,d \bx+\int_{\mathbb{T}^n}\bv \,d\bx=\0\quad\text{and}\quad \frac{d}{d t}\int_{\mathbb{T}^n}\bb\,d \bx=\0.
\end{align*}
Then,
\begin{align*}
	\int_{\mathbb{T}^n}\bv
\,d \bx=e^{-t}\int_{\mathbb{T}^n}\bv _0\,d \bx=\0\quad\text{and}\quad \int_{\mathbb{T}^n}\bb
\,d \bx=\int_{\mathbb{T}^n}\bb_0
\,d \bx=\0.
\end{align*}
This completes the proof.
\end{proof}
\end{lemma}

In the following, we give a lemma involving fractional operators which act on zero-mean functions. Note that the Fourier convention on the torus $\mathbb{T}^n=[0,2\pi]^n$ is given by
\[
\widehat \bv(\mathbf{j})=\frac{1}{(2\pi)^n}\int_{\mathbb{T}^n} \bv(\bx)e^{-i \mathbf{j}\cdot \bx}\,{d}\bx,\qquad
\bv(\bx)=\sum_{\mathbf{j}\in\mathbb Z^n}\widehat \bv(\mathbf{j})e^{i \mathbf{j}\cdot \bx}.
\]
\begin{lemma}\label{lem:zero-mean-fractional}
Let \( \bv \in \mathcal{S}'(\mathbb{T}^n) \) be a tempered distribution satisfying
\[
\int_{\mathbb{T}^n} \bv(\bx)\,{d}\bx = \0.
\]
The operator \(\Lambda^s\) is defined in the Fourier sense through
\[
\widehat{\Lambda^s \bv}(\mathbf{j})=|\mathbf{j}|^{s}\,\widehat \bv(\mathbf{j}),
\]
then,  for any \( s\in\mathbb R\), one gets
\[
\int_{\mathbb{T}^n} \Lambda^s \bv(\bx)\,{d}\bx = \0.
\]
\end{lemma}
\begin{proof}
The condition $\int_{\mathbb{T}^n} \bv(\bx)\,{d}\bx = \0$ implies \(\widehat \bv(\0)=\0\), then
\[
\bv(\bx)=\sum_{\bj\in\mathbb Z^n\setminus\{\0\}}\widehat \bv(\bj)e^{i \bj\cdot \bx},
\]
which means
\[
\Lambda^s \bv(\bx)=\sum_{\bj\in\mathbb Z^n\setminus\{\0\}}|\bj|^s\widehat \bv(\bj)e^{i \bj\cdot \bx}.
\]
Integrating on $\mathbb{T}^n$ yields
\[
\int_{\mathbb{T}^n}\Lambda^s \bv(\bx)\,{d}\bx
   =\sum_{\bj\in\mathbb Z^n\setminus\{\0\}}|\bj|^s\widehat \bv(\bj)
     \int_{\mathbb{T}^n} e^{i \bj\cdot \bx}\,{d}\bx.
\]
As \(\int_{\mathbb{T}^n} e^{i \bj\cdot \bx}\,{d}\bx=0\) for each \(\bj\ne\0\),  then
\[
\int_{\mathbb{T}^n}\Lambda^s \bv(\bx)\,{d}\bx=0.
\]
This finishes  the proof of Lemma \ref{lem:zero-mean-fractional}.
\end{proof}

In what follows, we give the following Poincar\'{e}-type inequality under the Diophantine condition.
\begin{lemma}{\rm (\cite{chen-2022-3dmhd-Diophant,Jiu-Liu-Xie-compressibleMHD})}\label{2.1}
Let $\widetilde{\bb}\in\mr^n$ be a given vector fulfilling the Diophantine condition \eqref{Diophantine}. For any $s\in\mr$, there exists a constant $c$ such that if $g\in H^{s+r+1}(\mathbb{T}^n)$ satisfies $\int_{\mathbb{T}^n}g d \bx=0$, then
	\begin{align}\label{D1}
		\|g\|_{H^s(\mathbb{T}^n)}\leq c \|\widetilde{\bb}\cdot\nabla g\|_{H^{s+r}(\mathbb{T}^n)}.
	\end{align}
\end{lemma}

At last, we list two frequently-used estimates. Their proofs can be found in \cite{kato-1988-CommutatorEstimatesEuler, kenig}.
\begin{lemma}\label{jiaohuanzi}
Let \( l > 0 \), \( 1 \leq p, p_1, p_2, q_1, q_2 \leq \infty \) and \( \frac{1}{p} = \frac{1}{p_1} + \frac{1}{q_1} = \frac{1}{p_2} + \frac{1}{q_2} \). Then, there exists a constant \( C>0 \) such that
\begin{itemize}
    \item for any \( f_1 \in W^{1,p_1} \cap W^{l,q_2} \) and \( f_2 \in L^{p_2} \cap W^{l-1,q_1} \),
    \[
    \| \Lambda^l (f_1f_2) - f_1 \Lambda^l f_2 \|_{L^p} \leq C \big( \| \nabla f_1 \|_{L^{p_1}} \| \Lambda^{l-1} f_2 \|_{L^{q_1}} + \| f_2 \|_{L^{p_2}} \| \Lambda^l f \|_{L^{q_2}} \big);
    \]
    \item for \( f_1 \in L^{p_1} \cap W^{l,q_2} \) and \( f_2 \in L^{p_2} \cap W^{l,q_1} \),
    \[
    \| \Lambda^l (f_1f_2) \|_{L^p} \leq C \big( \| f \|_{L^{p_1}} \| \Lambda^l f_2 \|_{L^{q_1}} + \| f_2 \|_{L^{p_2}} \| \Lambda^l f_1 \|_{L^{q_2}} \big).
    \]
\end{itemize}
\end{lemma}

\section{Kernel estimates} \label{sec:kernel}

In this section, we establish the  integral estimates for the solution to \eqref{equation} and get upper bounds for the corresponding kernel functions by spectral analysis.  Throughout this section, we leverage  \(\langle \cdot, \cdot \rangle\) to denote the standard inner product on \(\mathbb{C}^n\)\,(\(n \geq 2\)).

\begin{prop}\label{pro1}
Let $(\bv,\bb)$ be a solution to system \eqref{equation}. Then $(\bv,\bb)$ is capable of being bounded as
\begin{align}
\left\{\begin{aligned}\label{2.4}
&|\widehat{\bv}|\leq\left(|\widehat{G}_2|+
|\widehat{G}_3|\right)|\widehat{\boldsymbol{\psi}}_0|
 +\int_0^t\left(|\widehat{G}_2(t-\tau)||\widehat{\mathbf{N}}(\tau)|
+|\widehat{G}_3(t-\tau)||\widehat{\mathbf{N}}_1(\tau)|\right)d\tau,\\[1ex]
&|\widehat{\bb}|\leq\left(|\widehat{G}_1|+
|\widehat{G}_3|\right)|\hat{\boldsymbol{\psi}}_0| +\int_0^t\left(|\widehat{G}_1(t-\tau)||\widehat{\mathbf{N}}(\tau)|
+|\widehat{G}_3(t-\tau)||\widehat{\mathbf{N}}_2(\tau)|\right)d\tau,
\end{aligned}\right.
\end{align}
where $\boldsymbol{\psi}_0:=(\bv_0, \bb_0)^T$, $\mathbf{N}=({\mathbf{N}_1}, {\mathbf{N}_2})^T$, ${\mathbf{N}_1}:=\mathbb{P}(\bv \cdot \nabla \bv-\bb \cdot \nabla \bb)$, ${\mathbf{N}_2}:=\bv \cdot \nabla \bb- \bb \cdot \nabla \bv
$ and $\mathbb{P}$ denotes the Helmholtz-Leray projection operator. The kernel functions $\widehat{G}_1,\, \widehat{G}_2$ and $\widehat{G}_3$ are expressed by
\begin{align}\label{2.5'}
&|\widehat{G}_1|
=\frac{\left| e^{-\lambda_{2}t} - e^{-\lambda_{1}t} \right|}{|\lambda_{1} - \lambda_{2}|} \sqrt{ |\lambda_{1}|^2 + |\widetilde{\bb} \cdot \bj|^2 }=:|\widehat{G}|\sqrt{ |\lambda_{1}|^2 + |\widetilde{\bb} \cdot \bj|^2 },\notag\\[1ex]
&|\widehat{G}_2|
=\frac{\left| e^{-\lambda_{2}t} - e^{-\lambda_{1}t} \right|}{|\lambda_{1} - \lambda_{2}|} \sqrt{ |\lambda_{1}|^2 + |\widetilde{\bb} \cdot \bj|^2 }\left| \frac{\widetilde{\bb}\cdot \bj}{\lambda_{1}} \right|=:|\widehat{G}|\sqrt{ |\lambda_{1}|^2 + |\widetilde{\bb} \cdot \bj|^2 }\left| \frac{\widetilde{\bb}\cdot \bj}{\lambda_{1}} \right|, \\[1ex]
&\widehat{G}_3 =e^{-\lambda_{1}t},\notag
\end{align}
where $\widehat{G}$ represents
\begin{align}\label{2.5''}
\widehat{G} := \frac{e^{-\lambda_{2}t} - e^{-\lambda_{1}t}}{\lambda_{1} - \lambda_{2}}
\end{align}
with $\lambda_1$ and $\lambda_2$ denoting the roots of the  characteristic equation
\begin{align}\label{2.7}
\lambda^2 - \lambda + |\widetilde{\bb}\cdot \bj|^2=0,
\end{align}
i.e.,
\begin{align*}
\lambda_{1}(\bj) = \frac{1+ \sqrt{1 - 4|\widetilde{\bb}\cdot \bj|^2}}{2},\,\,\,\,\,\lambda_{2}(\bj) = \frac{1 - \sqrt{1 - 4|\widetilde{\bb}\cdot \bj|^2}}{2}.
\end{align*}
When $\lambda_1 = \lambda_2$, \eqref{2.4} still holds if we replace $\widehat{G}$ in \eqref{2.5''} by
\begin{equation*}
\widehat{G}=\lim_{\lambda_2 \to \lambda_1} \frac{e^{-\lambda_{2}t} - e^{-\lambda_{1}t}}{\lambda_{1} - \lambda_{2}}=te^{-\lambda_{1}t}.
\end{equation*}
\end{prop}
\begin{proof}[ Proof of Proposition {\rm \ref{pro1}} ]
Take the Helmholtz-Leray projection operator
$\mathbb{P}:=\text{Id} - \nabla \Delta^{-1} \nabla \cdot$
to the velocity equation in $\eqref{equation}_1$ to get
\begin{equation}\label{1.41}
\left\{ \begin{array}{l}
{\partial _t}\bv+ \bv -\widetilde{\bb} \cdot \nabla \bb+ \mathbf{N}_1(\bv, \bb)=\0,\\[1ex]
{\partial _t}\bb-\widetilde{\bb} \cdot \nabla \bv+\mathbf{N}_2(\bv, \bb) =\0,\\[1ex]
\mathrm{div}\,\bv = \mathrm{div}\,\bb = 0,\\[1ex]
(\bv,\bb)|_{t=0} = (\bv_0, \bb_0),
\end{array} \right.
\end{equation}
where
\begin{equation}\label{1.411}
\mathbf{N}_1(\bv, \bb)=\mathbb{P}(\bv \cdot \nabla \bv-\bb \cdot \nabla \bb),\,\,\,\,\,\,\,
\mathbf{N}_2(\bv, \bb)=\bv \cdot \nabla \bb-  \bb \cdot \nabla \bv.
\end{equation}
Executing  the Fourier transform to \eqref{1.41}, we have  for \(\bj \in \mathbb{Z}^n \setminus \{\0\}\) that
\begin{equation}\label{97}
\left\{ \begin{array}{l}
\partial_t \widehat{\bv} +  \widehat{\bv} - i (\widetilde{\bb} \cdot \bj) \widehat{\bb}  + \left( I - \frac{\bj \otimes \bj}{|\bj|^2} \right) \left( \widehat{\bv \cdot \nabla \bv} - \widehat{\bb \cdot \nabla \bb} \right)=\0,\\[2ex]
\partial_t \widehat{\bb} -i(\widetilde{\bb}\cdot \bj)\widehat{\bv}+ \widehat{\bv\cdot\nabla \bb} - \widehat{\bb\cdot\nabla \bv} =\0.
\end{array} \right.
\end{equation}
Rewrite the system in vector form \(\widehat{\boldsymbol{\psi}} := (\widehat{\bv}, \widehat{\bb})^T\) to obtain
\begin{align}\label{b}
\partial_t \hat{\boldsymbol{\psi}} + \mathbf{Q} \hat{\boldsymbol{\psi}} + \widehat{\mathbf{N}}(\bv, \bb) = \0,
\end{align}
where
\begin{align}\label{1.412}
\mathbf{Q}: =
\begin{pmatrix}
1 & -i\widetilde{\bb}\cdot \bj \\
-i\widetilde{\bb}\cdot \bj & 0
\end{pmatrix}, \quad
\widehat{\mathbf{N}}(\bv, \bb) :=
\begin{pmatrix}
\widehat{\mathbf{N}_1} \\
\widehat{\mathbf{N}_2}
\end{pmatrix}.
\end{align}
By virtue of  Duhamel's principle, the solution to \eqref{b} is written as
\begin{align}\label{cd}
\widehat{\boldsymbol{\psi}}&=e^{-\mathbf{Q}t}\widehat{\boldsymbol{\psi}}_0
-\int_0^te^{-\mathbf{Q}(t-\tau)}\widehat{\mathbf{N}}(\bv, \bb)(\tau)\, d\tau.
\end{align}
As the characteristic polynomial of the matrix \( \mathbf{Q} \) can be expressed as
\[
\det(\mathbf{Q} - \lambda \mathbf{I}) = \lambda^2 - \lambda + |\widetilde{\bb}\cdot \bj|^2,
\]
it has two eigenvalues \( \lambda_{1,2}(\bj) \) with corresponding eigenvectors \( \mathbf{a}_{1,2}(\bj) \), defined by
\[
\lambda_{1,2}(\bj) = \frac{1 \pm \sqrt{1 - 4|\widetilde{\bb}\cdot \bj|^2}}{2}, \quad
\mathbf{a}_{1,2}(\bj) = \begin{pmatrix} \lambda_{1,2} \\ -i\widetilde{\bb}\cdot \bj \end{pmatrix}.
\]

Denote
\[
\mathbf{A}_1 := (\mathbf{a}_1 \ \mathbf{a}_2), \quad
\mathbf{A}_2 := \mathbf{A}_1^{-1} = \frac{1}{\lambda_1 - \lambda_2} \begin{pmatrix}
1 & \frac{\lambda_2}{i\widetilde{\bb}\cdot \bj} \\
-1 & -\frac{\lambda_1}{i\widetilde{\bb}\cdot \bj}
\end{pmatrix} = \begin{pmatrix} \overline{\mathbf{b}}_1^{T} \\ \overline{\mathbf{b}}_2^{T} \end{pmatrix}.
\]
 Thus, we can diagonalize the matrix \( \mathbf{Q} \) as
\[
 \mathbf{Q}= \mathbf{A}_1 \begin{pmatrix} \lambda_1 & 0 \\ 0 & \lambda_2 \end{pmatrix} \mathbf{A}_2 = \lambda_1 \mathbf{a}_1 \overline{\mathbf{b}}_1^{T} + \lambda_2 \mathbf{a}_2 \overline{\mathbf{b}}_2^{T},
\]
and the matrix exponential \( e^{-\mathbf{Q}t} \) can be expressed as
\[
e^{-\mathbf{Q} t} = e^{-\lambda_1 t} \mathbf{a}_1 \overline{\mathbf{b}}_1^{T} + e^{-\lambda_2 t} \mathbf{a}_2 \overline{\mathbf{b}}_2^{T}.
\]
Then, \eqref{cd} becomes
\begin{align}\label{c}
\widehat{\boldsymbol{\psi}}&=e^{-\mathbf{Q}t}\widehat{\boldsymbol{\psi}}_0
-\int_0^te^{-\mathbf{Q}(t-\tau)}\widehat{\mathbf{N}}\, d\tau\notag\\[1ex]
&=e^{-\lambda_1 t} \mathbf{a}_1 \overline{\mathbf{b}}_1^{T}\widehat{\boldsymbol{\psi}}_0 + e^{-\lambda_2 t} \mathbf{a}_2 \overline{\mathbf{b}}_2^{T}\widehat{\boldsymbol{\psi}}_0
-\int_0^t\left( e^{-\lambda_1 (t-\tau)} \mathbf{a}_1 \overline{\mathbf{b}}_1^{T} + e^{-\lambda_2 (t-\tau)} \mathbf{a}_2 \overline{\mathbf{b}}_2^{T}\right)\widehat{\mathbf{N}}(\tau)\, d\tau\notag\\[1ex]
&=e^{-\lambda_1 t} \langle \widehat{\boldsymbol{\psi}}_0, \mathbf{b}_1 \rangle \mathbf{a}_1 + e^{-\lambda_2 t} \langle \widehat{\boldsymbol{\psi}}_0, \mathbf{b}_2 \rangle \mathbf{a}_2\notag\\[1ex]
&\quad-\int_0^t e^{-\lambda_1 (t-\tau)}\langle \widehat{\mathbf{N}}(\tau),\mathbf{b}_{1}\rangle\mathbf{a}_{1}\, d\tau-\int_0^t e^{-\lambda_2 (t-\tau)}\langle \widehat{\mathbf{N}}(\tau),\mathbf{b}_{2}\rangle\mathbf{a}_{2}\, d\tau,
\end{align}
where we used the fact \( \mathbf{l} \overline{\mathbf{m}}^{T} \mathbf{n} = \mathbf{l}\langle \mathbf{n}, \mathbf{m} \rangle =\langle \mathbf{n}, \mathbf{m} \rangle \mathbf{l} \)  for any vectors \( \mathbf{l} \), \( \mathbf{m} \) and \( \mathbf{n} \).
However,  we find from \eqref{c} that, in the neighborhood of the set \(\{1=4|\widetilde{\bb}\cdot \bj|^{2}\}\), the vectors  $\mathbf{b}_{1}$ and $\mathbf{b}_{2}$ have singularity. To eliminate this singularity, we rewrite $\widehat{\boldsymbol{\psi}}$ as
\begin{align*}
\widehat{\boldsymbol{\psi}}=&(e^{-\lambda_{2}t}-e^{-\lambda_{1}t})
\langle\hat{\boldsymbol{\psi}}_0,\mathbf{b}_{2}\rangle\mathbf{a}_{2}
+e^{-\lambda_{1}t}\hat{\boldsymbol{\psi}}_0\\[1ex]
&-\int_0^t(e^{-\lambda_{2}(t-\tau)}-e^{-\lambda_{1}(t-\tau)})
\langle \widehat{\mathbf{N}}(\tau),\mathbf{b}_{2}\rangle\mathbf{a}_{2}\, d\tau
-\int_0^te^{-\lambda_{1}(t-\tau)}\widehat{\mathbf{N}}(\tau)\, d\tau.
\end{align*}
Then,   we have for \( \widehat{\bb} \) that
\begin{align}\label{3.1}
\widehat{\bb}= &\langle \widehat{\boldsymbol{\psi}}, \mathbf{e}_2 \rangle\notag\\[1ex] =&(e^{-\lambda_{2}t}-e^{-\lambda_{1}t})
\langle\hat{\boldsymbol{\psi}}_0,\mathbf{b}_{2}\rangle\langle\mathbf{a}_{2}, \mathbf{e}_2\rangle
+e^{-\lambda_{1}t}\langle\hat{\boldsymbol{\psi}}_0,\mathbf{e}_2\rangle\notag\\[1ex]
&-\int_0^t(e^{-\lambda_{2}(t-\tau)}-e^{-\lambda_{1}(t-\tau)})
\langle \widehat{\mathbf{N}}(\tau),\mathbf{b}_{2}\rangle\langle\mathbf{a}_{2}, \mathbf{e}_2\rangle
 \, d\tau-\int_0^te^{-\lambda_{1}(t-\tau)}\langle \widehat{\mathbf{N}}(\tau), \mathbf{e}_2\rangle \, d\tau.
\end{align}
Similarly, for \( \widehat{\bv} \), we find
\[
\langle \mathbf{a}_2, \mathbf{e}_1 \rangle = \frac{\lambda_2}{-i\widetilde{\bb}\cdot \bj} \langle \mathbf{a}_2, \mathbf{e}_2 \rangle = \frac{i\widetilde{\bb}\cdot \bj}{\lambda_1} \langle \mathbf{a}_2, \mathbf{e}_2 \rangle,
\]
then
\begin{align}\label{3.2}
\widehat{\bv}= &\langle \widehat{\boldsymbol{\psi}}, \mathbf{e}_1 \rangle\notag\\[1ex] =& \frac{i\widetilde{\bb}\cdot \bj}{\lambda_1} (e^{-\lambda_2 t} - e^{-\lambda_1 t}) \langle \widehat{\boldsymbol{\psi}}_0, \mathbf{b}_2 \rangle \langle \mathbf{a}_2, \mathbf{e}_2 \rangle + e^{-\lambda_1 t} \langle \widehat{\boldsymbol{\psi}}_0, \mathbf{e}_1 \rangle\notag\\[1ex]
&-\int_0^t\frac{i\widetilde{\bb}\cdot \bj}{\lambda_1}(e^{-\lambda_{2}(t-\tau)}-e^{-\lambda_{1}(t-\tau)})
\langle\widehat{\mathbf{N}}(\tau),\mathbf{b}_{2}\rangle\langle\mathbf{a}_{2}, \mathbf{e}_2\rangle \, d\tau-\int_0^te^{-\lambda_{1}(t-\tau)}\langle \widehat{\mathbf{N}}(\tau), \mathbf{e}_1\rangle\, d\tau.
\end{align}
Combining \eqref{3.1} and \eqref{3.2}, we deduce
\begin{align*}
|\widehat{\bb}|\leq& \left\lvert\left( e^{-\lambda_{2}t}-e^{-\lambda_{1}t}\right) \langle\widehat{\boldsymbol{\psi}}_0,\mathbf{b}_{-}\rangle\langle\mathbf{a}_{2}, \mathbf{e}_2\rangle\right\rvert+ \left\lvert e^{-\lambda_{1}t}\widehat{\bb}_0 \right\rvert\notag\\[1ex]
&+\int_0^t\left\lvert\left(e^{-\lambda_{2}(t-\tau)}-e^{-\lambda_{1}(t-\tau)}\right)
\langle \widehat{\mathbf{N}}(\tau),\mathbf{b}_{-}\rangle\langle\mathbf{a}_{2}, \mathbf{e}_2\rangle\right\lvert  \,d\tau+\int_0^t\left\lvert e^{-\lambda_{1}(t-\tau)}\widehat{\mathbf{N}}_2(\tau) \right\lvert \,d\tau,
\end{align*}
and
\begin{align*}
|\widehat{\bv}|\leq& \left\lvert\frac{\widetilde{\bb}\cdot \bj}{\lambda_{1}}\left( e^{-\lambda_{2}t}-e^{-\lambda_{1}t}\right) \langle\hat{\boldsymbol{\psi}}_0,\mathbf{b}_{2}\rangle\langle\mathbf{a}_{2}, \mathbf{e}_2\rangle\right\rvert+ \left\lvert e^{-\lambda_{1}t}\widehat{\bv}_0 \right\rvert\notag\\[1ex]
&+\int_0^t\left\lvert\frac{\widetilde{\bb}\cdot \bj}{\lambda_{1}}\left(e^{-\lambda_{2}(t-\tau)}-e^{-\lambda_{1}(t-\tau)}\right)
\langle \widehat{\mathbf{N}}(\tau),\mathbf{b}_{2}\rangle\langle\mathbf{a}_{2}, \mathbf{e}_2\rangle\right\lvert \, d\tau+\int_0^t\left\lvert e^{-\lambda_{1}(t-\tau)}\widehat{\mathbf{N}}_1(\tau)\right\lvert\, d\tau.
\end{align*}
By the definitions of $\mathbf{a}_{2}$ and $\mathbf{b}_{2}$,  for any $\mathbf{f}\in \mathbb{C}^2$, we infer
\begin{align*}%\label{A}
\left| (e^{-\lambda_{2}t} - e^{-\lambda_{1}t}) \langle \mathbf{f}, \mathbf{b}_{2} \rangle \langle \mathbf{a}_{2}, \mathbf{e}_2 \rangle \right|\leq &\left| e^{-\lambda_{2}t} - e^{-\lambda_{1}t} \right| |\mathbf{b}_{2}| |\langle \mathbf{a}_{2}, \mathbf{e}_2 \rangle| |\mathbf{f}|\notag\\[1ex]
=&{\left| e^{-\lambda_{2}t} - e^{-\lambda_{1}t} \right|}\frac{\sqrt{ |\lambda_{1}|^2 + |\widetilde{\bb} \cdot \bj|^2 }}{|\lambda_{1} - \lambda_{2}|} |\mathbf{f}|\notag\\[1ex]
=&:|\widehat{G}_1(\bj,t)||\mathbf{f}|,
\end{align*}
and
\begin{align*}%\label{A}
\left| \frac{\widetilde{\bb}\cdot \bj}{\lambda_{1}}(e^{-\lambda_{2}t} - \mathbf{e}^{-\lambda_{1}t}) \langle \mathbf{f}, \mathbf{b}_{2} \rangle \langle \mathbf{a}_{2}, \mathbf{e}_2 \rangle \right|
\leq& \left| e^{-\lambda_{2}t} - e^{-\lambda_{1}t} \right| |\mathbf{b}_{2}| |\langle \mathbf{a}_{2}, \mathbf{e}_2 \rangle|\left| \frac{\widetilde{\bb}\cdot \bj}{\lambda_{2}} \right| |\mathbf{f}|\notag\\[1ex]
=&\frac{\left| e^{-\lambda_{2}t} - e^{-\lambda_{1}t} \right|}{|\lambda_{1} - \lambda_{2}|} {\sqrt{ |\lambda_{1}|^2 + |\widetilde{\bb} \cdot \bj|^2 }}\left| \frac{\widetilde{\bb}\cdot \bj}{\lambda_{1}} \right||\mathbf{f}|\notag\\[1ex]
=&:|\widehat{G}_2(\bj,t)||\mathbf{f}|.
\end{align*}
Define
$$
\widehat{G}_3(\bj,t): = e^{-\lambda_{1}t},\quad  \widehat{G}(\bj,t) := \frac{e^{-\lambda_{2}t} - e^{-\lambda_{1}t}}{\lambda_{1} - \lambda_{2}}.
$$
Then
$$
|\widehat{G}_1(\bj,t)| = |\widehat{G}(\bj,t)|{\sqrt{ |\lambda_{1}|^2 + |\widetilde{\bb} \cdot \bj|^2 }},\quad |\widehat{G}_2(\bj,t)| = |\widehat{G}(\bj,t)|{\sqrt{ |\lambda_{1}|^2 + |\widetilde{\bb} \cdot \bj|^2 }}\left| \frac{\widetilde{\bb}\cdot \bj}{\lambda_{1}} \right|.
$$
Therefore,
\begin{align*}
|\widehat{\bb}|\leq|\widehat{G}_1||\hat{\boldsymbol{\psi}}_0|+
|\widehat{G}_3||\widehat{\bb}_0| +\int_0^t|\widehat{G}_1(t-\tau)||\widehat{\mathbf{N}}(\tau)|\, d\tau
+\int_0^t|\widehat{G}_3(t-\tau)||\widehat{\mathbf{N}_2}(\tau)|\, d\tau,
\end{align*}
and
\begin{align*}
|\widehat{\bv}|\leq|\widehat{G}_2||\hat{\boldsymbol{\psi}}_0|+
|\widehat{G}_3||\widehat{\bv}_0| +\int_0^t|\widehat{G}_2(t-\tau)||\widehat{\mathbf{N}}(\tau)|\, d\tau
+\int_0^t|\widehat{G}_3(t-\tau)||\widehat{\mathbf{N}_1}(\tau)|\,d\tau.
\end{align*}
We finish  the proof of Proposition \ref{pro1}.
\end{proof}

In what follows, we investigate the behavior of the kernel functions $\widehat{K_1}$--$\widehat{K_3}$. To this end,  we decompose the frequency space $\mathbb{Z}^n \setminus \{\0\}$ into three subsections.
\begin{prop}\label{pro2}
Splitting  the frequency space $\mathbb{Z}^n\setminus\{\0\}$ into the following three subregions:
\begin{align*}
&S_1 := \left\{ \bj \in \mathbb{Z}^n \setminus \{\0\} : 1 - 4|\widetilde{\bb}\cdot \bj|^2 \leq 0 \right\},\notag\\[1ex]
&S_2 := \left\{\bj \in \mathbb{Z}^n \setminus \{\0\} : 0 < 1 - 4|\widetilde{\bb}\cdot \bj|^2 \leq \tfrac{1}{4} \right\},\notag\\[1ex]
&S_3 := \left\{ \bj \in \mathbb{Z}^n \setminus \{\0\} : 1- 4|\widetilde{\bb}\cdot \bj|^2 > \tfrac{1}{4} \right\},
\end{align*}
then, there exists an absolute constant $C > 0$ such that
\[
\begin{cases}
|\widehat{G}_1(\bj,t)|, \quad |\widehat{G}_2(\bj,t)| \leq C|\bj|e^{-\frac{t}{4}}, & \bj \in S_1, \\
|\widehat{G}_1(\bj,t)|, \quad |\widehat{G}_2(\bj,t)| \leq Ce^{-\frac{t}{8}}, &  \bj \in S_2, \\
|\widehat{G}_1(\bj,t)| \leq Ce^{-|\widetilde{\bb} \cdot \bj|^2 t}, \quad |\widehat{G}_2(\bj,t)| \leq C|\widetilde{\bb} \cdot \bj|e^{-|\widetilde{\bb} \cdot \bj|^2 t}, &  \bj \in S_3,
\\
 |\widehat{G}_3|\leq e^{-\frac{t}{2}}, &  \bj \in \mathbb{Z}^n \setminus \{\0\}.
\end{cases}
\]
\end{prop}

\begin{proof}[Proof of Proposition {\rm \ref{pro2}}]
To estimate \( |\widehat{G}_1| \), \( |\widehat{G}_2| \), and \( |\widehat{G}_3| \), we  proceed in steps.

\medskip
\noindent
\textbf{Step 1: Estimate for $\widehat{G}(\bj,t)$.} Recall that
\begin{align*}
\widehat{G}(\bj,t) := \frac{e^{-\lambda_{2}t} - e^{-\lambda_{1}t}}{\lambda_{1} - \lambda_{2}}.
\end{align*}
In the following we estimate $\widehat{G}(\bj,t)$ for each region $S_1$, $S_2$, and $S_3$.

\medskip
\noindent
\underline{\textbf{Case 1: $\bj \in S_1$}}. Let
\[
\sigma := \sqrt{4|\widetilde{\bb}\cdot \bj|^2 - 1},
\]
so that $\lambda_{1,2} = \frac{1}{2} \pm i \frac{\sigma}{2}$. Then, making use of Euler's formula,
\[
\widehat{G}(\bj,t) = e^{-\frac{t}{2}} \cdot \frac{e^{i\frac{\sigma}{2}t} - e^{-i\frac{\sigma}{2}t}}{i\sigma}
= t e^{-\frac{t}{2}} \cdot \frac{\sin\left( \frac{\sigma}{2}t \right)}{\frac{\sigma}{2}t}.
\]
The fact $\left|\frac{\sin \theta}{\theta}\right| \leq 1$ for all $\theta \in \mathbb{R}$ means
\[
|\widehat{G}(\bj,t)| \leq t e^{-\frac{t}{2}}.
\]
The classical inequality
\begin{align}\label{ye}
z^n e^{-z} \leq C_n, \quad \text{for all } z>0,\; n \in \mathbb{N},
\end{align}
with $z = \frac{t}{4}$, $n=1$, gives rise to
\[
|\widehat{G}(k,t)| \leq C e^{-\frac{t}{4}}.
\]

\medskip
\noindent
\underline{\textbf{Case 2: $\bj \in S_2$}}.  Note
\begin{align}
&-\frac{3}{4}\leq-\lambda_1=\frac{-1-\sqrt{1-4|\widetilde{\bb}\cdot \bj|^{2}}}{2}<-\frac{1}{2},\label{S_2}\\
&-\frac{1}{2}<-\lambda_2 =\frac{-1+\sqrt{1-4|\widetilde{\bb}\cdot \bj|^{2}}}{2}\leq-\frac{1}{4}.\notag
\end{align}
By the standard mean-value theorem, there exists a point $\zeta \in (-\lambda_1,-\lambda_2)$ such that
\[
|\widehat{G}(\bj,t)| = \left| \frac{e^{-\lambda_2 t} - e^{-\lambda_1 t}}{\lambda_1 - \lambda_2} \right| = t e^{\zeta t}.
\]
The fact $\zeta \leq -\frac{1}{4}$ implies
\[
|\widehat{G}(\bj,t)| \leq t e^{-\frac{1}{4} t} \leq C e^{-\frac{t}{8}}.
\]

\medskip
\noindent
\underline{\textbf{Case 3: $\bj \in S_3$}}. The root $-\lambda_1$ fulfills
\begin{align}\label{S_3}
-1 \leq -\lambda_1 = -\frac{1}{2} \left(1 + \sqrt{1 - 4{(\widetilde{\bb} \cdot \bj)^2}} \right) < -\frac{3}{4}.
\end{align}
While the root $-\lambda_2$ may go to zero and  we rewrite it as
\[
-\lambda_2 = -\frac{1}{2} \left(1 - \sqrt{1 - 4{(\widetilde{\bb} \cdot \bj)^2}} \right)
= \frac{-2(\widetilde{\bb} \cdot \bj)^2}{1 + \sqrt{1 - 4(\widetilde{\bb} \cdot \bj)^2}}.
\]
As $\bj \in S_3$ means $\frac{1}{4} <1 - 4(\widetilde{\bb}\cdot \bj)^2 \leq 1$, one has
\[
-\frac{4(\widetilde{\bb}\cdot \bj)^2}{3} < -\lambda_2 \leq -(\widetilde{\bb}\cdot \bj)^2, \quad
\lambda_1 - \lambda_2 = \sqrt{1 - 4|\widetilde{\bb}\cdot \bj|^2} \geq \frac{1}{2}.
\]
Therefore,
\[
|\widehat{G}(\bj,t)| \leq C\left({e^{-\frac{3}{4} t} + e^{-(\widetilde{\bb}\cdot \bj)^2}}\right)
\leq C{e^{-(\widetilde{\bb}\cdot \bj)^2t}}.
\]

\medskip
\noindent
\textbf{Step 2: Estimates for \( |\widehat{G}_1| \) and \( |\widehat{G}_2| \).}
Recalling from \eqref{2.5'} that
\begin{gather}
\begin{split}\label{G_12}
|\widehat{G}_1|^2 &= |\widehat{G}|^2 \left( |\lambda_1|^2 + |\widetilde{\bb}\cdot \bj|^2 \right), \\
|\widehat{G}_2|^2 &= |\widehat{G}|^2 \left( |\lambda_1|^2 + |\widetilde{\bb}\cdot \bj|^2 \right) \left| \frac{\widetilde{\bb}\cdot \bj}{\lambda_1} \right|^2,
\end{split}
\end{gather}
then we need to estimate the quantities \( |\lambda_1|^2 + |\widetilde{\bb}\cdot \bj|^2 \) and \( \left| \frac{\widetilde{\bb}\cdot \bj}{\lambda_1} \right|^2 \) in each region.

\medskip
\noindent
\underline{\textbf{Estimate of \( |\lambda_1|^2 + |\widetilde{\bb}\cdot \bj|^2 \)}}:

\smallskip
\noindent
\textit{For \(\bj \in S_1\)}: Since \( 1 - 4|\widetilde{\bb}\cdot \bj|^2 \leq 0 \), there exists a constant \( C \) such that
\[
|\lambda_1|^2 + |\widetilde{\bb}\cdot \bj|^2 = 2|\widetilde{\bb}\cdot \bj|^2 \leq C|\bj|^2.
\]

\smallskip
\noindent
\textit{For \(\bj \in S_2 \cup S_3\)}: Combing  \eqref{S_2} and \eqref{S_3}, we know that \( |\lambda_1| \leq C \), and
\[
|\lambda_1|^2 + |\widetilde{\bb}\cdot \bj|^2 \leq C.
\]

\medskip
\noindent
\underline{\textbf{Estimate of \( \left| \frac{\widetilde{\bb}\cdot \bj}{\lambda_1}\right|^2 \)}}:

\smallskip
\noindent
\textit{For \(\bj \in S_1\)}: By the definition of $\lambda_1$, one has
\begin{align}\label{S1_1}
\left| \frac{\widetilde{\bb}\cdot \bj}{\lambda_1} \right|^2 \leq C.
\end{align}

\smallskip
\noindent
\textit{For \(\bj \in S_2\)}: \eqref{S_2} implies \( |\lambda_1| \in \left[ \frac{1}{2}, \frac{3}{4} \right] \), and the fact \( 0 < 1 - 4|\widetilde{\bb}\cdot \bj|^2 \) means \eqref{S1_1} also is true.

\smallskip
\noindent
\textit{For \(\bj \in S_3\)}: From \eqref{S_3}, we have \( |\lambda_1| \in \left[ \frac{3}{4}, 1 \right] \), then
\[
\left| \frac{\widetilde{\bb}\cdot \bj}{\lambda_1} \right| \leq C|\widetilde{\bb}\cdot \bj|.
\]

\medskip
\noindent
Combining the above, we obtain
\begin{equation}\label{eq:lambda-bounds}
|\lambda_1|^2 + |\widetilde{\bb}\cdot \bj|^2 \leq
\begin{cases}
C|\bj|, & \bj \in S_1, \\
C, & \bj \in S_2 \cup S_3,
\end{cases}
\quad
\left| \frac{\widetilde{\bb}\cdot \bj}{\lambda_1} \right|^2 \leq
\begin{cases}
C, & \bj \in S_1 \cup S_2, \\
|\widetilde{\bb}\cdot \bj|^2, & \bj \in S_3.
\end{cases}
\end{equation}
Inserting  the bounds for $|\widehat{G}(\bj,t)|$ derived in \textbf{Step 1} and \eqref{eq:lambda-bounds} into \eqref{G_12}, we obtain
\[
\begin{cases}
|\widehat{G}_1(\bj,t)|, \quad |\widehat{G}_2(\bj,t)| \leq C|\bj|e^{-\frac{t}{4}}, & \bj \in S_1, \\
|\widehat{G}_1(\bj,t)|, \quad |\widehat{G}_2(\bj,t)| \leq Ce^{-\frac{t}{8}}, &  \bj \in S_2, \\
|\widehat{G}_1(\bj,t)| \leq Ce^{-|\widetilde{\bb}\cdot \bj|^2 t}, \quad |\widehat{G}_2(\bj,t)| \leq C|\widetilde{\bb}\cdot \bj|e^{-|\widetilde{\bb}\cdot \bj|^2 t}, &  \bj \in S_3.
\end{cases}
\]

\medskip
\noindent
\textbf{Step 3: Estimate for \( |\widehat{G}_3| \).}  Since \( \Re \lambda_1 = \frac{1}{2} \) for all \( \bj \neq \0 \), we get
\[
|\widehat{G}_3| \leq e^{-\frac{t}{2}}, \quad \text{for all } \bj \in \mathbb{Z}^n \setminus \{\0\}.
\]

\smallskip
\noindent
\textit{For \(\bj \in S_1\)},
\[
|\widehat{G}_3| = \left| e^{-\lambda_1 t} \right| = \left| e^{-\frac{t}{2}} \cdot e^{-i\frac{\sqrt{4|\widetilde{\bb}\cdot \bj|^2 - 1}}{2}} \right| \leq e^{-\frac{t}{2}}.
\]

\smallskip
\noindent
\textit{For \(\bj \in  S_2 \cup S_3\)},
\[
|\widehat{G}_3| = e^{-\frac{t}{2}} \cdot e^{-\frac{\sqrt{1 - 4|\widetilde{\bb}\cdot \bj|^2}}{2} t} \leq e^{-\frac{t}{2}}.
\]
Thus, for any \( \bj \in \mathbb{Z}^n \setminus \{\0\} \), we conclude that there exists a constant \( C \) such that
\[
|\widehat{G}_3| \leq C e^{-\frac{t}{2}}.
\]
This ends the proof of Proposition \ref{pro2}.
\end{proof}
\section{Linear stability}\label{2/3}

In this section, we establish the linear stability and derive the time-decay estimates for system \eqref{lineartheorem}.
Recall the linearized equations
\begin{equation}\label{linearrepeat}
\left\{
\begin{array}{l}
{\partial _t}\bV + \bV - \widetilde{\bb} \cdot \nabla \bH = \0,\\[1ex]
{\partial _t}\bH - \widetilde{\bb} \cdot \nabla \bV = \0,\\[1ex]
\mathrm{div}\,\bV = \mathrm{div}\,\bH = 0,\\[1ex]
(\bV,\bH)|_{t = 0} = (\bV_0, \bH_0).
\end{array}
\right.
\end{equation}

\begin{proof}[Proof of Theorem {\rm\ref{thm1}}]
  Operating the operator $\Lambda^s$ to $\eqref{linearrepeat}_1$ and $\eqref{linearrepeat}_2$, executing the $L^2$ inner products with $\Lambda^s \bV$ and $\Lambda^s \bH$, respectively, and summing the results, we have
\begin{align*}
\frac{1}{2} \frac{d}{d t}\left(\|\bV\|_{\dot{H}^s}^2+\|\bH\|_{\dot{H}^s}^2\right)
+ \|\bV\|_{\dot{H}^s}^2
= \int_{\mathbb{T}^n} (\widetilde{\bb}\cdot\nabla \Lambda^s \bH) \cdot \Lambda^s \bV \,d x
+ \int_{\mathbb{T}^n} (\widetilde{\bb}\cdot\nabla \Lambda^s \bV)\cdot \Lambda^s \bH \,d x.
\end{align*}
By integration by parts and using $\mathrm{div}\,\bV=\mathrm{div}\,\bH=0$, one easily checks that
\begin{align*}
\int_{\mathbb{T}^n} (\widetilde{\bb}\cdot\nabla \Lambda^s \bH) \cdot \Lambda^s \bV \,d x
+ \int_{\mathbb{T}^n} (\widetilde{\bb}\cdot\nabla \Lambda^s \bV)\cdot \Lambda^s \bH \,d x = 0.
\end{align*}
Consequently,
\begin{align}\label{gpp.4.3}
\frac{1}{2}\frac{d}{d t}\left(\|\bV\|_{\dot{H}^s}^2+\|\bH\|_{\dot{H}^s}^2\right)
+ \|\bV\|_{\dot{H}^s}^2 = 0.
\end{align}
Integrating \eqref{gpp.4.3} in time gives, for any $0 \le s \le m$,
\begin{align}\label{gpp.4.4}
\|\bV(t)\|_{\dot{H}^s}^2+\|\bH(t)\|_{\dot{H}^s}^2
\le \|\bV_0\|_{\dot{H}^s}^2+\|\bH_0\|_{\dot{H}^s}^2.
\end{align}
The above energy inequality ensures that for any $m\ge0$, the solution $(\bV,\bH)$ to \eqref{linearrepeat} remains uniformly bounded in time.
To derive decay rates, we consider the Fourier representation provided by Proposition~\ref{pro1}, which reads
\begin{align}\label{new1}
\begin{cases}
|\widehat{\bV}(\bj,t)|^2 \le \big(|\widehat{G}_2(\bj,t)|^2 + |\widehat{G}_3(\bj,t)|^2\big)
|\widehat{\boldsymbol{\phi}}_0(\bj)|^2, \\[1ex]
|\widehat{\bH}(\bj,t)|^2 \le \big(|\widehat{G}_1(\bj,t)|^2 + |\widehat{G}_3(\bj,t)|^2\big)
|\widehat{\boldsymbol{\phi}}_0(\bj)|^2,
\end{cases}
\end{align}
where $\boldsymbol{\phi}_0 := (\bV_0, \bH_0)^T$.
Combining \eqref{new1} with the kernel estimates in Proposition~\ref{pro2} and assuming \eqref{mean-zero}, one can establish the time-decay properties of the Sobolev norms of $\bV$ and $\bH$.

Indeed, since \eqref{mean-zero} ensures that for all $t\ge0$,
\begin{equation}\label{mean-zero2}
\widehat{\bV}(\0,t) = \widehat{\bH}(\0,t) = \0,
\end{equation}
the zero-frequency component vanishes.
Applying \eqref{new1}, \eqref{mean-zero2}, and the Plancherel theorem,  for any $0 \le s \le m$,  we infer
\begin{align}\label{U}
\lVert \bV(\bx, t)\rVert^2_{\dot{H}^{s}}\leq \sum_{\bj\neq \0}\lvert \bj\rvert^{2s}|\widehat{G}_2|^2|\widehat{\boldsymbol{\phi}}_0|^2
+\sum_{\bj\neq \0}\lvert \bj\rvert^{2s}|\widehat{G}_3|^2|\widehat{\boldsymbol{\phi}}_0|^2=:I_2+I_3,
\end{align}
and
\begin{align*}
\lVert \bH(\bx, t)\rVert^2_{\dot{H}^{s}}&=\sum_{\bj \in \mathbb{Z}^n \setminus \{\0\}} |\bj|^{2s} |\widehat{\bH}(\bj, t)|^2\\
&\leq \sum_{\bj\neq \0}\lvert \bj\rvert^{2s}|\widehat{G}_1|^2|\widehat{\boldsymbol{\phi}}_0|^2+
\sum_{\bj\neq \0}\lvert \bj\rvert^{2s}|\widehat{G}_3|^2|\widehat{\boldsymbol{\phi}}_0|^2=:I_1+I_3.
\end{align*}
Proposition \ref{pro2} implies
\begin{align}\label{25}
I_3= \sum_{\bj\neq \0}\lvert \bj\rvert^{2s}|\widehat{G}_3|^2|\widehat{\boldsymbol{\phi}}_0|^2\leq  e^{-t}  \sum_{\bj\neq \0} \lvert \bj\rvert^{2s} \lvert \widehat{\boldsymbol{\phi}}_0\rvert^{2}\leq e^{-t}\lVert{\boldsymbol{\phi}}_0\rVert^2_{\dot{H}^{s}}.
\end{align}
For the term $I_1$, Proposition \ref{pro2} and the Diophantine condition \eqref{Diophantine} yield
\begin{align}
I_1 &\leq C \sum_{\bj \in S_1} |\bj|^{2(s+1)} e^{-\frac{t}{2}} |\widehat{\boldsymbol{\phi}}_0|^2 + C \sum_{\bj \in S_2} |\bj|^{2s} e^{-\frac{t}{4}} |\widehat{\boldsymbol{\phi}}_0|^2 + C \sum_{\bj \in S_3} e^{-2|\widetilde{\bb} \cdot \bj|^2 t} |\bj|^{2s}|\widehat{\boldsymbol{\phi}}_0|^2\notag\\
&\leq Ce^{-\frac{t}{2}}\|{\boldsymbol{\phi}}_0(\bx)\|_{\dot{H}^{s+1}}^2
+Ce^{-\frac{t}{4}}\|{\boldsymbol{\phi}}_0(\bx)\|_{\dot{H}^{s}}^2+C\sum_{\bj\in S_3}
e^{-\frac{2c^2}{|\bj|^{2r}} t}\lvert \bj\rvert^{2s} \lvert\hat{\boldsymbol{\phi}}_0\rvert^{2}.\notag
\end{align}
By virtue of \eqref{ye}, the third term can be bounded as
\begin{align}
	&\sum_{\bj\in S_3}e^{-\frac{2c^2}{|\bj|^{2r}} t}\lvert \bj\rvert^{2s} |\widehat{\boldsymbol{\phi}}_0|^2\notag\\
	=&\sum_{\bj\in S_3}e^{-\frac{2c^2}{|\bj|^{2r}} t}\left(\frac{t}{|\bj|^{2r}}\right)^{\frac{m-s}{r}}t^{-\frac{m-s}{r}}|\bj|^{{2(m -s)}}\lvert \bj\rvert^{2s} |\widehat{\boldsymbol{\phi}}_0|^2\notag\\
	\leq&  t^{-\frac{m-s}{r}} \|{\boldsymbol{\phi}}_0(\bx)\|_{\dot{H}^{m}}^2\sup_{\bj\in S_3}e^{-\frac{2c^2}{|\bj|^{2r}} t}\left(\frac{t}{|\bj|^{2r}}\right)^{\frac{m-s}{r}}\notag\\
	\leq& Ct^{-\frac{m-s}{r}} \|{\boldsymbol{\phi}}_0(\bx)\|_{\dot{H}^{m}}^2,\notag
\end{align}	
and
\begin{align}
e^{-\frac{2c^2}{|\bj|^{2r}} t}\lvert \bj\rvert^{2s} |\widehat{\boldsymbol{\phi}}_0|^2\leq \lvert \bj\rvert^{2s} |\widehat{\boldsymbol{\phi}}_0|^2.\notag
\end{align}
It follows
\begin{align}%\label{S3}
\sum_{\bj\in S_3}e^{-\frac{2c^2}{|\bj|^{2r}} t}\lvert \bj\rvert^{2s} |\widehat{\boldsymbol{\phi}}_0|^2
	\leq& C(1+t)^{-\frac{m-s}{r}} \|{\boldsymbol{\phi}}_0(\bx)\|_{\dot{H}^{m}}^2.\notag
\end{align}
Through the above estimates, we deduce
\begin{align}\label{24}
\| \bH(t) \|_{\dot{H}^s}^2
\le Ce^{-\frac{t}{2}}\|{\boldsymbol{\phi}}_0\|_{\dot{H}^{s+1}}^2
+Ce^{-\frac{t}{4}}\|{\boldsymbol{\phi}}_0\|_{\dot{H}^{s}}^2
+ C (1+t)^{-\frac{m-s}{r}} \, \| \boldsymbol{\phi}_0 \|_{\dot{H}^m}^2.
\end{align}
The  estimate of \(I_2\) is given by
\begin{align}
I_2 &\leq C \sum_{\bj \in S_1} |\bj|^{2(s+1)} e^{-\frac{t}{2}} |\widehat{\boldsymbol{\phi}}_0|^2 + C \sum_{\bj\in S_2} |\bj|^{2s} e^{-\frac{t}{4}} |\widehat{\boldsymbol{\phi}}_0|^2 + C \sum_{\bj \in S_3}|\widetilde{\bb} \cdot \bj|^2 e^{-2|\widetilde{\bb} \cdot \bj|^2 t} |\bj|^{2s}|\widehat{\boldsymbol{\phi}}_0|^2.\notag
\end{align}
By virtue of the Diophantine condition \eqref{Diophantine} and \eqref{ye}, the third term admits
\begin{align}
	& \sum_{\bj \in S_3}|\widetilde{\bb} \cdot \bj|^2 e^{-2|\widetilde{\bb} \cdot \bj|^2 t} |\bj|^{2s}|\widehat{\boldsymbol{\phi}}_0|^2\notag\\
	=&\sum_{\bj\in S_3}\left({\lvert\widetilde{\bb}\cdot \bj\rvert^{2}t}\right)t^{-1}
e^{-2\lvert \widetilde{\bb}\cdot \bj\rvert^{2}t}\left( {{\lvert\widetilde{\bb}\cdot \bj\rvert^{2}}t} \right)^{\frac{m-s}{r}}\left( {{\lvert\widetilde{\bb}\cdot \bj\rvert^{2}}t} \right)^{-\frac{m-s}{r}}
\lvert \bj\rvert^{2s} \lvert\widehat{\boldsymbol{\phi}}_0\rvert^{2}\notag\\
	\leq&\sum_{\bj\in S_3}\left({\lvert\widetilde{\bb}\cdot \bj\rvert^{2}t}\right)t^{-1}
e^{-2\lvert \widetilde{\bb}\cdot bj\rvert^{2}t}\left( {{\lvert\widetilde{\bb}\cdot \bj\rvert^{2}}t} \right)^{\frac{m-s}{r}}t^{-\frac{m-s}{r}}|\bj|^{2r\frac{m-s}{r}}
\lvert \bj\rvert^{2s} \lvert\widehat{\boldsymbol{\phi}}_0\rvert^{2}\notag\\
=&\sum_{\bj\in S_3} t^{-\left(1+\frac{m-s}{r}\right)}e^{-{2\lvert \widetilde{\bb}\cdot \bj\rvert^{2}}t}\left( {{\lvert\widetilde{\bb}\cdot \bj\rvert^{2}}t} \right)^{\left(1+\frac{m-s}{r}\right)}
\lvert \bj\rvert^{2m}\lvert\widehat{\boldsymbol{\phi}}_0\rvert^{2}\notag\\
	\leq&  t^{-\left(1+\frac{m-s}{r}\right)} \|{\boldsymbol{\phi}}_0\|_{\dot{H}^{m}}^2\sup_{\bj\in S_3}e^{-{2\lvert \widetilde{\bb}\cdot \bj\rvert^{2}}t}\left( {{\lvert\widetilde{\bb}\cdot \bj\rvert^{2}}t} \right)^{\left(1+\frac{m-s}{r}\right)}\notag\\
	\leq& Ct^{-\left(1+\frac{m-s}{r}\right)} \|{\boldsymbol{\phi}}_0\|_{\dot{H}^{m}}^2.\notag
\end{align}
Since $\widetilde{\bb}$ is fixed, one has
\begin{align}
{|\widetilde{\bb}\cdot \bj|^2} e^{-{2\lvert\widetilde{\bb}\cdot \bj\rvert^{2}}t} \lvert \bj\rvert^{2s} \lvert\widehat{\boldsymbol{\phi}}_0\rvert^{2}\leq \lvert \bj\rvert^{2s} |\widehat{\boldsymbol{\phi}}_0|^2.\notag
\end{align}
Therefore
\begin{align*}
\sum_{\bj \in S_3}|\widetilde{\bb} \cdot \bj|^2 e^{-2|\widetilde{\bb} \cdot \bj|^2 t} |\bj|^{2s}|\widehat{\boldsymbol{\phi}}_0|^2&\leq C (1 + t)^{-\left(1+\frac{m-s}{ r}\right)} \|{\boldsymbol{\phi}}_0\|_{\dot{H}^{m}}^2.
\end{align*}
Thus, we get
\begin{align*}
\| \bV(t) \|_{\dot{H}^s}^2
\le Ce^{-\frac{t}{2}}\|{\boldsymbol{\phi}}_0\|_{\dot{H}^{s+1}}^2
+Ce^{-\frac{t}{4}}\|{\boldsymbol{\phi}}_0\|_{\dot{H}^{s}}^2
+ C (1 + t)^{-\left(1+\frac{m-s}{ r}\right)} \| \boldsymbol{\phi}_0 \|_{\dot{H}^m}^2,
\end{align*}
which, together with  \eqref{gpp.4.4} and \eqref{24}, finishes the proof of Theorem~\ref{thm1}.
\end{proof}

\section{Nonlinear stability}\label{s:6}
This section is devoted to the nonlinear stability  presented in Theorem~\ref{thm}. For clarity, we divide the processes into three parts.

\subsection{A priori estimates}\label{4/1}
 The local well-posedness of system \eqref{equation} could be derived by the classical approaches, such as the Friedrichs mollifier or Fourier cutoff techniques; see, for example, \cite{li-2017-local,fefferman-2014-local,Majda-Bertozzi}. So it suffices to present uniform \emph{a priori} bounds to extend the local solution globally in time.

We first set the following modified energy functional
\begin{align}
	Q_s(t) := a \lVert (\bv, \bb)(t) \rVert_{H^s}^2
	- \sum_{l = 0}^{s} \int_{\mathbb{T}^n} (\widetilde{\bb} \cdot \nabla \bb)(t) \cdot \Lambda^{2l-2} \bv(t) \, d\bx,
	\quad s\in[0, m],\label{lya}
\end{align}
where \(a > 0\) is a suitably chosen constant. We shall focus on the evolution of $Q_s(t)$.

We begin with handling the first term on the right-hand side of \eqref{lya}.
\begin{prop}\label{prop4.2}
Let $(\bv,\bb)$ be a smooth global solution to $\eqref{equation}$.
Let $m > 2 + r + \frac{n}{2}$. For any $s \in [0, m]$ and $t \in [0, T]$, it follows that
\begin{align}\label{2.33}
\frac{1}{2}\frac{d}{dt} \lVert( \Lambda^s \bv, \Lambda^s \bb)\rVert_{L^2}^2 +\lVert \Lambda^{s} \bv \rVert_{L^2}^2
\leq C \left( \lVert \nabla \bv \rVert_{L^{\infty}} + \lVert \bb \rVert_{H^{m-1-r}}^2 \right) \lVert( \Lambda^s \bv, \Lambda^s \bb)\rVert_{L^2}^2.
\end{align}
Particularly,
\begin{align}\label{2.33'}
\frac{1}{2}\frac{d}{dt} \lVert (\bv, \bb)(t) \rVert_{H^m}^2 +\lVert \bv \rVert_{H^m}^2
\leq C \left( \lVert \nabla \bv \rVert_{L^{\infty}} + \lVert \bb \rVert_{H^{m-1-r}}^2 \right) \lVert (\bv, \bb)(t) \rVert_{H^m}^2.
\end{align}
\end{prop}
\begin{proof}
 Executing the operator \(\Lambda^s\) to both sides of $\eqref{equation}_1$ and $\eqref{equation}_2$, operating the inner products of the results with \(\Lambda^s\bv\) and \(\Lambda^s\bb\), respectively, and then adding them together, we deduce
\[
\frac{1}{2} \frac{d}{dt} \left( \lVert \Lambda^s \bv \rVert_{L^2}^2 + \lVert \Lambda^s \bb \rVert_{L^2}^2 \right) + \lVert \Lambda^{s} \bv \rVert_{L^2}^2  = \sum_{j = 1}^{6} \mathrm{I}_j,
\]
where
\begin{align*}
\mathrm{I}_1 &= - \int_{\mathbb{T}^n} \Lambda^s (\bv \cdot \nabla \bv) \cdot \Lambda^s \bv \, d\bx, & \mathrm{I}_2 &= \int_{\mathbb{T}^n} \Lambda^s (\bb \cdot \nabla \bb) \cdot \Lambda^s \bv \, d\bx, \\[1ex]
\mathrm{I}_3 &= - \int_{\mathbb{T}^n} \Lambda^s (\bv \cdot \nabla \bb) \cdot \Lambda^s \bb \, d\bx, & \mathrm{I}_4 &= \int_{\mathbb{T}^n} \Lambda^s (\bb \cdot \nabla \bv) \cdot \Lambda^s \bb \, d\bx, \\[1ex]
\mathrm{I}_5 &= \int_{\mathbb{T}^n} (\widetilde{\bb} \cdot \nabla \Lambda^s \bb) \cdot \Lambda^s \bv \, d\bx, & \mathrm{I}_6 &= \int_{\mathbb{T}^n} (\tilde{b} \cdot \nabla \Lambda^s \bv) \cdot \Lambda^s \bb \, d\bx.
\end{align*}
Through direct computations and integrating by parts, we infer
\[\mathrm{I}_5 + \mathrm{I}_6 = \int_{\mathbb{T}^n} (\widetilde{\bb} \cdot \nabla \Lambda^s \bb) \cdot \Lambda^s \bv \, dx + \int_{\mathbb{T}^n} (\widetilde{\bb} \cdot \nabla \Lambda^s \bv) \cdot \Lambda^s \bb \, d\bx = 0.
\]
Applying Lemma \ref{jiaohuanzi} and \eqref{equation}$_3$, one has
\begin{align*}
\mathrm{I}_1 = - \int_{\mathbb{T}^n} \Lambda^s (\bv \cdot \nabla \bv) \cdot \Lambda^s \bv \, d\bx &= - \int_{\mathbb{T}^n} \left[ \Lambda^s (\bv \cdot \nabla \bv) - \bv \cdot \nabla \Lambda^s \bv \right] \cdot \Lambda^s \bv \, d\bx \\[1ex]
&\leq C \lVert \Lambda^s \bv \rVert_{L^2}^2 \lVert \nabla \bv \rVert_{L^{\infty}},
\end{align*}
\begin{align*}
\mathrm{I}_3 = - \int_{\mathbb{T}^n} \Lambda^s (\bv \cdot \nabla \bb) \cdot \Lambda^s \bb \, d\bx &= - \int_{\mathbb{T}^n} \left[ \Lambda^s (\bv \cdot \nabla \bb) - \bv \cdot \nabla \Lambda^s \bb \right] \cdot \Lambda^s \bb \, d\bx \\[1ex]
&\leq C \lVert \Lambda^s \bb \rVert_{L^2} \left( \lVert \Lambda^s \bv \rVert_{L^2} \lVert \nabla \bb \rVert_{L^{\infty}} + \lVert \Lambda^s \bb \rVert_{L^2} \lVert \nabla \bv \rVert_{L^{\infty}} \right),
\end{align*}
and
\begin{align*}
\mathrm{I}_2 + \mathrm{I}_4 &= \int_{\mathbb{T}^n} \Lambda^s (\bb \cdot \nabla \bb) \cdot \Lambda^s \bv \, d\bx + \int_{\mathbb{T}^n} \Lambda^s (\bb \cdot \nabla \bv) \cdot \Lambda^s \bb \, d\bx \\
&= \int_{\mathbb{T}^n} \left[ \Lambda^s (\bb \cdot \nabla \bb) - \bb \cdot \Lambda^s \nabla \bb \right] \cdot \Lambda^s \bv \, d\bx \\
&\quad + \int_{\mathbb{T}^n} \left[ \Lambda^s (\bb \cdot \nabla \bv) - \bb \cdot \Lambda^s \nabla \bv \right] \cdot \Lambda^s \bb \, d\bx \\[1ex]
&\leq C \lVert \Lambda^s \bv \rVert_{L^2} \lVert \Lambda^s \bb \rVert_{L^2} \lVert \nabla \bb \rVert_{L^{\infty}} \\[1ex]
&\quad + C \lVert \Lambda^s \bb \rVert_{L^2} \left( \lVert \Lambda^s \bv \rVert_{L^2} \lVert \nabla \bb \rVert_{L^{\infty}} + \lVert \Lambda^s \bb \rVert_{L^2} \lVert \nabla \bv \rVert_{L^{\infty}} \right).
\end{align*}
Making use of the Sobolev embedding
\begin{align}\label{2.55}
\lVert \nabla \bb \rVert_{L^{\infty}} \leq C \lVert \bb \rVert_{H^{m-1-r}}, \quad \text{with } m > 2 + r + \frac{n}{2},
\end{align}
which, together with the preceding estimates, yields
\begin{align*}
&\frac{1}{2} \frac{d}{dt} \left( \lVert \Lambda^s \bv \rVert_{L^2}^2 + \lVert \Lambda^s \bb \rVert_{L^2}^2 \right) + \lVert \Lambda^{s + 1} \bv \rVert_{L^2}^2\\[1ex]
 \leq &C \lVert \nabla \bv \rVert_{L^{\infty}} \left( \lVert \Lambda^s \bv \rVert_{L^2}^2 + \lVert \Lambda^s \bb \rVert_{L^2}^2 \right)+C\lVert \nabla \bb \rVert_{L^{\infty}}\lVert \Lambda^s u \rVert_{L^2} \lVert \Lambda^s \bb \rVert_{L^2} \\[1ex]
 \leq &C \lVert \nabla\bv \rVert_{L^{\infty}} \left( \lVert \Lambda^s \bv \rVert_{L^2}^2 + \lVert \Lambda^s \bb \rVert_{L^2}^2 \right)+C\lVert \bb \rVert_{H^{m-1-r}}\lVert \Lambda^s \bv \rVert_{L^2} \lVert \Lambda^s \bb \rVert_{L^2}.
\end{align*}
At last, the Young inequality results in
\begin{align*}
&\frac{1}{2} \frac{d}{dt} \left( \lVert \Lambda^s \bv \rVert_{L^2}^2 + \lVert \Lambda^s \bb \rVert_{L^2}^2 \right) + \lVert \Lambda^{s} \bv \rVert_{L^2}^2\\[1ex]
 \leq &C \lVert \nabla \bv \rVert_{L^{\infty}} \left( \lVert \Lambda^s \bv \rVert_{L^2}^2 + \lVert \Lambda^s \bb \rVert_{L^2}^2 \right)+\varepsilon\lVert \Lambda^{s} \bv \rVert^2_{L^2}+C_\varepsilon\lVert \bb \rVert^2_{H^{m-1-r}} \lVert \Lambda^s \bb \rVert^2_{L^2}.
\end{align*}
Absorbing the \( \varepsilon \)-term into the left-hand side yields the desired result. So far, we end the proof of Proposition~\ref{prop4.2}.
\end{proof}

Let us mention that, the norm $\lVert \bb \rVert_{H^{m-1-r}}$, which appears in Proposition~\ref{prop4.2}, plays an important role in estimating the nonlinear terms.
To exploit how dissipation acts on \( \bb \), in what follows,  we turn to the second term on the right-hand side of \eqref{lya}.
\begin{prop}\label{prop4.3}
Let $m > 2 + r + \frac{n}{2}$, and let $(\bv,\bb)$ be a smooth global solution to $\eqref{equation}$.
Then, for any $s \in [0, m]$ and $t \in [0, T]$, it holds that
\begin{align}\label{2.333}
-\frac{d}{dt} \sum_{l = 0}^{s}\int_{\mathbb{T}^n} (\widetilde{\bb} \cdot \nabla \bb) \cdot \Lambda^{2l-2} \bv \, d\bx
\leq &\left( 1+ |\widetilde{\bb}|^2 \right) \left\lVert \bv \right\rVert_{H^s}^2
-\frac{1}{2} \left\lVert \Lambda^{-1} (\widetilde{\bb} \cdot \nabla \bb) \right\rVert_{H^s}^2 \notag\\[1ex]
&+ C \left\lVert (\bv, \bb) \right\rVert_{H^s}^2 \left( \left\lVert \nabla \bv \right\rVert_{L^{\infty}} + \lVert \bb \rVert^2_{H^{m-1-r}} \right),
\end{align}
where $C > 0$ is a constant depending only on fixed parameters.
\end{prop}
\begin{proof}
Straightforward calculations give rise to
\begin{align*}
-\frac{d}{dt} \int_{\mathbb{T}^n} (\widetilde{\bb} \cdot \nabla \bb) \cdot \Lambda^{2l-2} \bv \, d\bx &= -\int_{\mathbb{T}^n} (\widetilde{\bb} \cdot \nabla \partial_t \bb) \cdot \Lambda^{2l-2} \bv \, d\bx - \int_{\mathbb{T}^n} (\widetilde{\bb} \cdot \nabla \bb) \cdot \Lambda^{2l-2} \partial_t \bv \, d\bx\\[1ex]
&=: \mathrm{II}_1 + \mathrm{II}_2.
\end{align*}
To proceed, we first focus on estimating \( \mathrm{II}_1 \). Due to its complexity, \( \mathrm{II}_1 \) is further divided into two parts, \( \mathrm{II}_{11} \) and \( \mathrm{II}_{12} \). After completing these estimates, we then turn our attention to \( \mathrm{II}_2 \).

\smallskip
\noindent
\textbf{Estimate of \(\mathrm{II}_1\):} Using $\eqref{equation}_2$, we rewrite it as
\begin{align*}
\mathrm{II}_1&=-\int_{\mathbb{T}^n} (\widetilde{\bb} \cdot \nabla (\widetilde{\bb} \cdot \nabla \bv + \bb \cdot \nabla \bv - \bv\cdot \nabla \bb)) \cdot \Lambda^{2l-2} \bv \, d\bx\\[1ex]
&=-\int_{\mathbb{T}^n} (\widetilde{\bb} \cdot \nabla) (\widetilde{\bb} \cdot \nabla \bv) \cdot \Lambda^{2l-2} \bv \, d\bx - \int_{\mathbb{T}^n} \widetilde{\bb} \cdot \nabla (\bb \cdot \nabla \bv - \bv \cdot \nabla \bb) \cdot \Lambda^{2l-2} \bv \, d\bx\\[1ex]
&=:\mathrm{II}_{11}+\mathrm{II}_{12}.
\end{align*}
For $\mathrm{II}_{11}$, thanks to  the fact
\[
(\widetilde{\bb} \cdot \nabla)\Lambda^{2l-2} = \Lambda^{2l-2}(\widetilde{\bb} \cdot \nabla) \quad \text{for all } l \in \mathbb{R},
\]
and Plancherel's theorem, we have
\begin{align*}
\mathrm{II}_{11} &= -\int_{\mathbb{T}^n} (\widetilde{\bb} \cdot \nabla) (\widetilde{\bb} \cdot \nabla \bv) \cdot \Lambda^{2l-2} \bv \, d\bx = \int_{\mathbb{T}^n} (\widetilde{\bb} \cdot \nabla \bv) \cdot (\widetilde{\bb} \cdot \nabla \Lambda^{2l-2} \bv) \, d\bx\\[1ex]
&= \int_{\mathbb{T}^n} (\widetilde{\bb} \cdot \nabla \bv) \cdot \Lambda^{2l-2}(\widetilde{\bb} \cdot \nabla \bv) \, d\bx = \lVert \Lambda^{l-1}(\widetilde{\bb} \cdot \nabla \bv) \rVert_{L^2}^2\\[1ex]
& \leq |\widetilde{\bb}|^2 \lVert \Lambda^{l} \bv \rVert_{L^2}^2.
\end{align*}
To deal with \( \mathrm{II}_{12} \), we consider the case \( l \geq 1 \) and the case \( l = 0 \), respectively.

\smallskip
\noindent
{\textbf{Case $l \geq 1$.}} Applying integration by parts, H\"{o}lder inequality, Lemma \ref{jiaohuanzi}, the Sobolev embedding \eqref{2.55} and Young inequality, we infer
\begin{align*}
\mathrm{II}_{12}
&= - \int_{\mathbb{T}^n} (\widetilde{\bb} \cdot \nabla)(\bb \cdot \nabla \bv - \bv \cdot \nabla \bb) \cdot \Lambda^{2l-2} \bv \, d\bx \\[1ex]
&= - \int_{\mathbb{T}^n} (\widetilde{\bb} \cdot \nabla)(\bb \cdot \nabla \bv - \bv \cdot \nabla \bb) \cdot \Lambda^{l-2}(\Lambda^{l} \bv) \, d\bx \\[1ex]
&\le C \|\widetilde{\bb} \cdot \nabla \Lambda^{l-2}(\bb \cdot \nabla \bv - \bv \cdot \nabla \bb)\|_{L^2} \|\Lambda^{l} \bv\|_{L^2} \\[1ex]
&\le C(|\widetilde{\bb}|) \big( \|\Lambda^{l-1} \bb\|_{L^2} \|\nabla \bv\|_{L^\infty} + \|\bb\|_{L^\infty} \|\Lambda^{l} \bv\|_{L^2} \big) \|\Lambda^{l} \bv\|_{L^2} \\[1ex]
&\quad + C(|\widetilde{\bb}|) \big( \|\Lambda^{l-1} \bv\|_{L^2} \|\nabla \bb\|_{L^\infty} + \|\bv\|_{L^\infty} \|\Lambda^{l} b\|_{L^2} \big) \|\Lambda^{l} \bv\|_{L^2} \\[1ex]
&\le C(|\widetilde{\bb}|) \|\nabla \bv\|_{L^\infty} \big( \|\Lambda^{l} \bv\|_{L^2}^2 + \|\Lambda^{l} \bb\|_{L^2}^2 \big)
+ \tfrac{1}{4} \|\Lambda^{l} \bv\|_{L^2}^2
+ C \|\bb\|_{H^{m - 1 - r}}^2 \|\Lambda^{l} \bv\|_{L^2}^2,
\end{align*}
where, in the last line, we have applied \eqref{meanzero2}, Plancherel's theorem  and the Fourier multiplier property \(|\bj|^{l-1}\leq |\bj|^{l}\) for all \(|\bj| \geq 1\), giving rise to the facts \(\lVert \Lambda^{l-1} \bv \rVert_{L^2} \leq \lVert \Lambda^{l} \bv \rVert_{L^2}\) and \(\lVert \Lambda^{l-1} \bb \rVert_{L^2} \leq \lVert \Lambda^{l} \bb \rVert_{L^2}\).

\smallskip
\noindent
{\textbf{Case $l = 0$.}} In this case, the term $\Lambda^{-2}\bv$ involving  a negative-order Sobolev norm, leads to Lemma \ref{jiaohuanzi} inapplicable. To solve this, we instead resort to the $L^2$-boundedness of the Riesz transforms $R_i = \partial_i \Lambda^{-1}$.
\begin{align*}
\mathrm{II}_{12}
&= - \int_{\mathbb{T}^n} (\widetilde{\bb} \cdot \nabla)(\bb \cdot \nabla \bv-\bv \cdot \nabla \bb) \cdot \Lambda^{-2} \bv \, d\bx \\[1ex]
&= - \int_{\mathbb{T}^n} \widetilde{b}_i \partial_i \Lambda^{-1}(b_j \partial_j v_k-v_j \partial_j b_k) \cdot \Lambda^{-1} v_k \, d\bx \\[1ex]
&= \int_{\mathbb{T}^n} \widetilde{b}_i R_i(b_j \partial_j v_k-v_j \partial_j b_k) \cdot \Lambda^{-1} v_k \, d\bx \\[1ex]
&\leq C \| \tilde{b} \|_{L^\infty}\left(\| R_i(b_j \partial_j v_k) \|_{L^2}+\| R_i (v_j \partial_j b_k) \|_{L^2}\right) \| \Lambda^{-1} v_k \|_{L^2} \\[1ex]
&\leq C(|\widetilde{\bb}|) \| \bb \cdot \nabla \bv \|_{L^2} \| \Lambda^{-1} \bv \|_{L^2}+C(|\widetilde{\bb}|)\|\bv \cdot \nabla \bb\|_{L^2}\| \Lambda^{-1} \bv\|_{L^2}\\[1ex]
&\leq C(|\widetilde{\bb}|) \| \nabla \bv \|_{L^\infty} \| \bb \|_{L^2} \| \bv \|_{L^2}+C(|\widetilde{\bb}|)\|\nabla \bb\|_{L^\infty}\| \bv\|_{L^2}\| \bv\|_{L^2} \\[1ex]
&\leq C(|\widetilde{\bb}|) \left\| \nabla \bv \right\|_{L^\infty} \left( \left\| \bv \right\|_{L^2}^2 + \left\| \bb \right\|_{L^2}^2 \right)+ \frac{1}{4}\left\| \bv \right\|_{L^2}^2 +  C\left\| \bb \right\|_{H^{m - 1 - r}}^2 \left\| \bv \right\|_{L^2}^2.
\end{align*}
Thus, for any $l\geq0$,
\begin{align}\label{8.1}
|\mathrm{II}_1|&\leq \left(|\widetilde{\bb}|^2+\frac{1}{4}\right) \lVert \Lambda^{l} \bv \rVert_{L^2}^2+ C\left(\left\| \nabla \bv \right\|_{L^\infty}+\left\| \bb \right\|_{H^{m - 1 - r}}^2 \right)  \left( \| \Lambda^{l} \bv\|_{L^2}^2 +\| \Lambda^{l} \bb \|_{L^2}^2 \right).
\end{align}

\noindent
\textbf{Estimate of \(\mathrm{II}_2\):} Considering the first equation in \eqref{1.41}, we integrate by parts to get
\begin{align*}
\mathrm{II}_2 &= -\int_{\mathbb{T}^n} (\widetilde{\bb} \cdot \nabla \bb) \cdot \Lambda^{2l-2}\left( \widetilde{\bb} \cdot \nabla \bb + \bv - \mathbb{P}(\bv \cdot \nabla \bv - \bb \cdot \nabla \bb) \right) \,d\bx \\[1ex]
&= -\lVert \Lambda^{l-1}(\widetilde{\bb} \cdot \nabla \bb) \rVert_{L^2}^2 - \int_{\mathbb{T}^n} (\widetilde{\bb} \cdot \nabla \bb) \cdot \Lambda^{2l-2} \bv \,d\bx\\[1ex]
 &\quad+ \int_{\mathbb{T}^n} (\widetilde{\bb} \cdot \nabla \bb) \cdot \Lambda^{2l-2} \mathbb{P}(\bv \cdot \nabla \bv - \bb \cdot \nabla \bb) \,d\bx \\[1ex]
&=: \mathrm{II}_{21} + \mathrm{II}_{22} + \mathrm{II}_{23}.
\end{align*}
For $\mathrm{II}_{22}$, thanks to H\"{o}lder and Young inequalities, one has
\begin{align*}
|\mathrm{II}_{22}|
&= \left| \int_{\mathbb{T}^n} (\widetilde{\bb} \cdot \nabla \bb) \cdot \Lambda^{2l-2} \bv\, d\bx \right| \\[1ex]
&\leq \lVert \Lambda^{l-1}(\widetilde{\bb} \cdot \nabla \bb) \rVert_{L^2} \,
      \lVert \Lambda^{l-1} \bv \rVert_{L^2} \\[1ex]
&\leq \varepsilon \lVert \Lambda^{l} \bv \rVert_{L^2}^2
   + \frac{1}{4\varepsilon} \lVert \Lambda^{l-1}(\widetilde{\bb} \cdot \nabla \bb) \rVert_{L^2}^2.
\end{align*}
Selecting $\varepsilon =\frac{3}{4}$ gives
\begin{align*}
|\mathrm{II}_{22}|
&\leq \frac{3}{4}\lVert \Lambda^{l} \bv \rVert_{L^2}^2
   + \frac{1}{3} \lVert \Lambda^{l-1}(\widetilde{\bb} \cdot \nabla \bb) \rVert_{L^2}^2.
\end{align*}
For $\mathrm{II}_{23}$, similar to $\mathrm{II}_{12}$, we split the cases $l \geq 1$  and $l = 0$.

\smallskip
\noindent
{\textbf{Case $l \geq 1$.}}
 Making use of the boundedness of the Leray projection \( \mathbb{P} \) on \( L^2 \), Lemma~\ref{jiaohuanzi}, the Sobolev embedding~\eqref{2.55}, Poincar\'e and Young inequalities, one deduces
\begin{align*}
|\mathrm{II}_{23}| &\leq \lVert \Lambda^{l-1}(\widetilde{\bb} \cdot \nabla \bb) \rVert_{L^2} \lVert \Lambda^{l-1}(\bv \cdot \nabla \bv - \bb \cdot \nabla \bb) \rVert_{L^2} \\[1ex]
&\leq C \lVert \Lambda^{l-1}(\widetilde{\bb} \cdot \nabla \bb) \rVert_{L^2} \left( \lVert \Lambda^{l} \bv \rVert_{L^2} \lVert \nabla \bv \rVert_{L^\infty} + \lVert \Lambda^{l} \bb \rVert_{L^2} \lVert \bb \rVert_{H^{m - 1 - r}} \right)\\[1ex]
&\leq C(|\widetilde{\bb}|)\lVert \Lambda^{l} \bb \rVert_{L^2} \lVert \Lambda^{l} \bv \rVert_{L^2}\lVert \nabla \bv \rVert_{L^{\infty}}+C\lVert \Lambda^{l-1} (\widetilde{\bb} \cdot \nabla \bb) \rVert_{L^2}\lVert \Lambda^{l} \bb \rVert_{L^2}\lVert \bb \rVert_{H^{m-1-r}}\\[1ex]
&\leq C(|\widetilde{\bb}|)\lVert \nabla \bv \rVert_{L^{\infty}}\left(\lVert \Lambda^{l} \bv \rVert^2_{L^2} +\lVert \Lambda^{l} \bb \rVert^2_{L^2}\right)+\frac{1}{6}\lVert \Lambda^{l-1} (\widetilde{\bb} \cdot \nabla \bb) \rVert_{L^2}^2 +  C\left\| \bb \right\|_{H^{m - 1 - r}}^2\|\Lambda^{l} \bb\|_{L^2}^2.
\end{align*}

\smallskip
\noindent
{\textbf{Case $l =0$.}}
In this case, we resort to direct \( L^2 \) estimates,
\begin{align*}
\mathrm{II}_{23}&= \int_{\mathbb{T}^n} (\widetilde{\bb} \cdot \nabla \bb) \cdot \Lambda^{-2} \mathbb{P}(\bv \cdot \nabla \bv - \bb \cdot \nabla \bb)\, d\bx\\[1ex]
&\leq \lVert \Lambda^{-2}(\widetilde{\bb} \cdot \nabla \bb) \rVert_{L^2}\left(\lVert \bv \cdot \nabla \bv \rVert_{L^2}+\lVert \bb \cdot \nabla \bb \rVert_{L^2}\right)  \\[1ex]
&\leq C \lVert \Lambda^{-1}(\widetilde{\bb} \cdot \nabla \bb) \rVert_{L^2} \left(\lVert \nabla \bv \rVert_{L^\infty}\lVert\bv \rVert_{L^2}+\lVert \nabla \bb \rVert_{L^\infty}\lVert \bb \rVert_{L^2}  \right)  \\[1ex]
&\leq C(|\widetilde{\bb}|) \lVert \bb \rVert_{L^2} \lVert \bv \rVert_{L^2}\lVert  \nabla \bv \rVert_{L^\infty}+C \lVert \Lambda^{-1}(\widetilde{\bb} \cdot \nabla \bb)\rVert_{L^2} \|\bb\|_{H^{m-1-r}} \lVert \bb \rVert_{L^2}\\[1ex]
&\leq C(|\widetilde{\bb}|)\lVert \nabla \bv \rVert_{L^{\infty}}\left(\lVert  \bv \rVert^2_{L^2} +\lVert \bb \rVert^2_{L^2}\right) +\frac{1}{6}\lVert  \Lambda^{-1} (\widetilde{\bb} \cdot \nabla \bb) \rVert_{L^2}^2 +  C\left\| \bb \right\|_{H^{m - 1 - r}}^2 \left\| \bb  \right\|_{L^2}^2.
\end{align*}
We conclude from the estimates of $\mathrm{II}_{21}, \mathrm{II}_{22}$ and $\mathrm{II}_{23}$ that,  for any $l\geq0$,
\begin{align}\label{8.2}
|\mathrm{II}_2|&\leq -\frac{1}{2} \lVert \Lambda^{l-1}(\widetilde{\bb} \cdot \nabla \bb) \rVert_{L^2}^2 +\frac{3}{4} \lVert \Lambda^{l+1} \bv \rVert_{L^2}^2\notag\\[1ex]
&\quad+ C\left(\lVert \nabla \bv \rVert_{L^{\infty}}+\left\| \bb \right\|_{H^{m - 1 - r}}^2\right)\left(\lVert \Lambda^{l} \bv \rVert^2_{L^2} +\lVert \Lambda^{l} \bb \rVert^2_{L^2}\right).
\end{align}
Therefore, from \eqref{8.1}, \eqref{8.2} and  the assumption \( m > 2 + r + \frac{n}{2} \), we infer
\begin{align*}
- \frac{d}{dt} \sum_{l = 0}^{s} \int_{\mathbb{T}^n} (\widetilde{\bb} \cdot \nabla \bb) \cdot \Lambda^{2l-2} \bv \, d\bx
\leq{}& \left( 1+ |\widetilde{\bb}|^2 \right) \left\lVert \bv \right\rVert_{H^s}^2
- \frac{1}{2} \left\lVert \Lambda^{-1} (\widetilde{\bb} \cdot \nabla \bb) \right\rVert_{H^s}^2 \\[1ex]
& + C \left\lVert (\bv, \bb) \right\rVert_{H^s}^2 \left( \left\lVert \nabla \bv \right\rVert_{L^{\infty}} + \left\lVert \bb \right\rVert_{H^{m-1-r}}^2 \right).
\end{align*}
So we finish the proof of Proposition~\ref{prop4.3}.
\end{proof}

With Propositions~\ref{prop4.2} and \ref{prop4.3} at hand, we proceed to get the energy estimates.
Set the energy functional as
\begin{align*}
E_m^2(T) :=  \sup_{t \in [0, T]} \lVert (\bv, \bb)(t) \rVert_{H^m}^2
+ \int_0^T \lVert \bv(t) \rVert_{H^m}^2 \, dt
+ \int_0^T \lVert \Lambda^{-1}(\widetilde{\bb} \cdot \nabla \bb)(t) \rVert_{H^m}^2 \, dt.
\end{align*}
Then, the following estimate holds.
\begin{prop}\label{prop3.211a}
Let \( n \in \mathbb{N} \) with \( n \geq 2 \), \(r>n-1\), and \( m \in \mathbb{N} \) fulfilling \( m > 2 + r + \frac{n}{2} \).
Suppose \( (\bv, \bb) \) is a smooth global solution to \eqref{equation}.
Then there exists a constant \( C > 0 \) such that
\begin{align}\label{12.24}
E_m^2(T) \leq C(|\widetilde{\bb}|) \lVert (\bv_0, \bb_0) \rVert_{H^m}^2
+ C E_m^4(T)
+ C E_m^2(T) \int_0^T \lVert \nabla \bv(t) \rVert_{L^{\infty}} \, dt,
\end{align}
for any \( T > 0 \).
\end{prop}
\begin{proof}
Note that the modified energy functional defined in \eqref{lya}:
\begin{align*}
Q_m(t):= a\lVert (\bv, \bb)(t) \rVert_{H^m}^2 - \sum_{l = 0}^{m} \int_{\mathbb{T}^n} (\widetilde{\bb} \cdot \nabla \bb)(t) \cdot \Lambda^{2l-2} \bv(t) \, d\bx.
\end{align*}
Multiplying \eqref{2.33'} by $a := 1+\tfrac{|\widetilde{\bb}|}{2}+\tfrac{|\widetilde{\bb}|^2}{2}$
and adding $\frac{1}{2}\eqref{2.333}$, we deduce
\begin{align}\label{12.2}
\frac{1}{2} \frac{d}{dt} Q_m(t)
\leq\; & -a \left\lVert \bv \right\rVert_{H^m}^2 - \frac{1}{4} \left\lVert \Lambda^{-1} (\widetilde{\bb} \cdot \nabla \bb) \right\rVert_{H^m}^2+ \left( \frac{1}{2} + \frac{|\widetilde{\bb}|^2}{2} \right) \left\lVert \bv \right\rVert_{H^m}^2 \notag\\
&  + C \lVert (\bv, \bb) \rVert_{H^m}^2 \left( \left\lVert \nabla \bv \right\rVert_{L^{\infty}} + \lVert \bb \rVert_{H^{m-1-r}}^2 \right) \notag\\
\leq\; & -\frac{1}{4} \left( \left\lVert \bv \right\rVert_{H^m}^2 + \left\lVert \Lambda^{-1} (\widetilde{\bb} \cdot \nabla \bb) \right\rVert_{H^m}^2 \right) + C \lVert (\bv, \bb) \rVert_{H^m}^2 \left( \left\lVert \nabla \bv \right\rVert_{L^{\infty}} + \lVert \bb \rVert_{H^{m-1-r}}^2 \right).
\end{align}
In addition,  Lemma~\ref{2.1} and the Poincar\'e inequality  yields
\begin{align*}%\label{12.2ab}
\|\bb\|_{H^{m-1-r}}
&\leq c \|\widetilde{\bb} \cdot \nabla \bb\|_{H^{m-1}}
\leq C \|\widetilde{\bb} \cdot \nabla \bb\|_{\dot{H}^{m-1}}\notag\\
& = C \left\|\Lambda^{m-1}(\widetilde{\bb} \cdot \nabla \bb)\right\|_{L^2}= C \left\|\Lambda^{-1}(\widetilde{\bb} \cdot \nabla \bb)\right\|_{\dot{H}^m} \notag\\
&\leq C \left\lVert \Lambda^{-1} (\widetilde{\bb} \cdot \nabla \bb) \right\rVert_{H^m},
\end{align*}
so that \eqref{12.2} is rewritten as
\begin{align}\label{12.233'}
&\frac{d}{dt} Q_m(t) + \frac{1}{2} \left( \left\lVert \bv \right\rVert_{H^m}^2 + \left\lVert \Lambda^{-1} (\widetilde{\bb} \cdot \nabla \bb) \right\rVert_{H^m}^2 \right)\notag\\
&\qquad\leq C \lVert (\bv, \bb) \rVert_{H^m}^2 \left( \left\lVert \nabla \bv \right\rVert_{L^{\infty}} + \left\lVert \Lambda^{-1} (\widetilde{\bb} \cdot \nabla \bb) \right\rVert_{H^m}^2 \right).
\end{align}
Thanks to H\"older, Poincar\'e and Young inequalities, we have
\begin{align*}
\left| \int_{\mathbb{T}^n} (\widetilde{\bb} \cdot \nabla \bb) \cdot \Lambda^{2l-2} \bv \, d\bx \right|
\leq |\widetilde{\bb}| \cdot \lVert \Lambda^l \bb \rVert_{L^2} \cdot \lVert \Lambda^{l - 1} \bv \rVert_{L^2} \leq \frac{|\widetilde{\bb}|}{2} \left( \lVert \Lambda^l \bb \rVert_{L^2}^2 + \lVert \Lambda^l \bv \rVert_{L^2}^2 \right),
\end{align*}
which means
\begin{align}\label{12.21}
\left| \sum_{l = 0}^{m} \int_{\mathbb{T}^n} (\widetilde{\bb} \cdot \nabla \bb) \cdot \Lambda^{2l-2} u \, dx \right|
\leq \frac{|\widetilde{\bb}|}{2} \lVert (\bv, \bb) \rVert_{H^m}^2.
\end{align}
Integrating \eqref{12.233'} over \([0, T]\), and combining with  \eqref{12.21} give rise to \eqref{12.24}.
\end{proof}

\subsection{Global-in-time existence}\label{4/2}
This subsection is devoted to establishing the global existence part of Theorem~\ref{thm}.
And our main task is to control  the critical term
\[
\int_0^T \|\nabla \bv(t)\|_{L^\infty} \, dt,
\]
appearing in the energy inequality \eqref{12.24}.
\begin{prop}\label{prop5.1}
Let \( n \in \mathbb{N} \) with \( n \ge 2 \), and let \( m \in \mathbb{N} \) so that \( m > 2 + \frac{n}{2} \). Suppose  $(\bv, \bb)$ is a smooth global solution to \eqref{equation}. Then there exists a constant $C > 0$ such that for all $T>0$,
\begin{align}\label{prop1.1}
 &\sum_{\bj \in \mathbb{Z}^n \setminus \{\0\}} \int_{0}^{T} |\bj|\left| \widehat{\bv}(t,\bj) \right| \, dt\notag\\[1ex]
 \leq& C\lVert \bv_0 \rVert_{H^m}
 + C\sup_{t \in [0, T]} \lVert \bv \rVert_{H^m}\sum_{\bj \in \mathbb{Z}^n \setminus \{\0\}}\int_{0}^{T} |\bj| |\widehat{\bv}(t,\bj)| \, dt\notag\\[1ex]
 & +C\sup_{t \in [0, T]} \lVert \bb \rVert_{H^m}\sum_{\bj \in \mathbb{Z}^n \setminus \{\0\}} \int_{0}^{T}| \widehat{\bb}(t,\bj)| \, dt+\sum_{\bj \in \mathbb{Z}^n \setminus \{\0\}} \int_{0}^{T} |\bj| |(\widetilde{\bb} \cdot \bj) \widehat{\bb} (t,\bj)| \, dt.
\end{align}
\end{prop}
\begin{proof}
Using Duhamel's principle to $\eqref{97}_1$ gives
\begin{align}\label{le'}
&\sum_{\bj \in \mathbb{Z}^n \setminus \{\0\}} \int_{0}^{T} |\bj| |\widehat{\bv}(t,\bj)| \, dt \leq J_1 + J_2 + J_3 + J_4,
\end{align}
where
\begin{align*}
&J_1:= \sum_{\bj \in \mathbb{Z}^n \setminus \{\0\}} \int_{0}^{T} |\bj| e^{- t} |\widehat{\bv}_0(\bj)| \, dt,\notag\\[1ex]
&J_2:= \sum_{\bj \in \mathbb{Z}^n \setminus \{\0\}} \int_{0}^{T} \int_{0}^{t} |\bj| e^{- (t - \tau)} |\widehat{\bv \cdot \nabla \bv}(\tau,\bj)| \, d\tau dt,\notag\\[1ex]
&J_3 := \sum_{\bj \in \mathbb{Z}^n \setminus \{\0\}} \int_{0}^{T} \int_{0}^{t} |\bj| e^{-(t - \tau)} |\widehat{\bb \cdot \nabla \bb} (\tau,\bj)| \, d\tau dt,\notag\\[1ex]
&J_4:= \sum_{\bj \in \mathbb{Z}^n \setminus \{\0\}} \int_{0}^{T} \int_{0}^{t} |\bj| e^{- (t - \tau)} |(\widetilde{\bb} \cdot \bj) \widehat{\bb} (\tau,\bj)| \, d\tau dt.
\end{align*}
Here we used the fact that $|\widehat{\mathbb{P} f}(\bj)| \leq |\widehat{f}(\bj)|$ when bounding $J_2$ and $J_3$.

Firstly, we deal with the term $J_1$. Noting that $\int_0^T e^{-t}\,\mathrm{d}t \le 1$, then
\[
J_1 \le \sum_{\bj \in \mathbb{Z}^n \setminus \{\0\}} |\bj|\,|\widehat{\bv}_0(\bj)|.
\]
Applying H\"{o}lder inequality and assuming $m > 1 + \frac{n}{2}$, we have
\begin{align*}
\sum_{\bj \in \mathbb{Z}^n \setminus \{\0\}} |\bj|\,|\widehat{\bv}_0(\bj)|
&\le
\Bigg( \sum_{\bj \in \mathbb{Z}^n \setminus \{\0\}} |\bj|^{-2m+2} \Bigg)^{1/2}
\Bigg( \sum_{\bj \in \mathbb{Z}^n \setminus \{\0\}} |\bj|^{2m} |\widehat{\bv}_0(\bj)|^2 \Bigg)^{1/2} \\[0.5ex]
&\le C\,\|\bv_0\|_{\dot{H}^m},
\end{align*}
where $C>0$ depends only on $n$ and $m$.

For $J_4$, we also have
\begin{align*}
J_4 \leq  \sum_{\bj \in \mathbb{Z}^n \setminus \{\0\}} \int_{0}^{T} |\bj| |(\widetilde{\bb} \cdot \bj) \widehat{\bb} (t,\bj)| \, dt.
\end{align*}
For \(J_2\), as $\nabla \cdot \bv = 0$, we  write
\[
|\widehat{\bv \cdot \nabla \bv}(t,\bj)|
= \big|\widehat{\nabla \cdot (\bv \otimes \bv)}(t,\bj)\big|
\leq |\bj|\,\big|(\widehat{\bv} * \widehat{\bv})(t,\bj)\big|,
\quad \bj \in \mathbb{Z}^n \setminus \{\0\}.
\]
Then we get that
\begin{align*}
J_{2} \leq  \sum_{\bj \in \mathbb{Z}^n \setminus \{\0\}} \int_{0}^{T} |\bj|^2\left|(\widehat{\bv}*\widehat{\bv}) (t,\bj)\right| \, d t.
\end{align*}
In what follows, using Young inequality for convolutions and  H\"{o}lder inequality, we obtain for \( m > 2 + \frac{n}{2} \),
\begin{align*}
J_{2}&\leq \int_{0}^{T} \sum_{\bj \in \mathbb{Z}^n \setminus \{\0\}} |\bj|^2|\widehat{\bv} (t,\bj)| \sum_{\bj \in \mathbb{Z}^n \setminus \{\0\}} |\widehat{\bv} (t,\bj)| \, dt\\[1ex]
&\leq  \sup_{t \in [0, T]} \sum_{\bj \in \mathbb{Z}^n \setminus \{\0\}} |\bj|^2|\widehat{\bv} (t,\bj)| \left( \sum_{\bj \in \mathbb{Z}^n \setminus \{\0\}} \int_{0}^{T}|\bj| |\widehat{\bv} (t,\bj)| \, dt \right)\\[1ex]
&\leq  C\sup_{t \in [0, T]} \| \bv(t) \|_{H^m}  \sum_{\bj \in \mathbb{Z}^n \setminus \{\0\}} \int_{0}^{T} |\bj| |\widehat{\bv} (t,\bj)| \, dt,
\end{align*}
here, as \(\bj \ne \0\), the Poincar\'e inequality has been used in the second line.

Similarly, for \(J_3\), it follows for \( m > 2 + \frac{n}{2} \)
 \begin{align*}
J_3 &\leq \sum_{\bj \in \mathbb{Z}^n \setminus \{\0\}} \int_{0}^{T}|\bj|^2 |(\widehat{\bb}*\widehat{\bb}) (t,\bj)| \, d t
\\[1ex]
&\leq C \sup_{t \in [0, T]} \left\lVert \bb  \right\rVert_{H^m}  \sum_{\bj \in \mathbb{Z}^n \setminus \{\0\}} \int_{0}^{T} |\widehat{\bb} (t,\bj)| \, dt.
\end{align*}
Combining the estimates \(J_1\), \(J_2\), \(J_3\) and \(J_4\) gives rise to \eqref{le'}, which finishes the proof of Proposition~\ref{prop5.1}.
\end{proof}

To close the estimate in Proposition~\ref{prop5.1}, it remains to bound the terms
\begin{align*}
 \sum_{\bj \in \mathbb{Z}^n \setminus \{\0\}}\int_0^T |\widehat{\bb}(t,\bj)|\,dt, \quad {\rm and} \, \sum_{\bj \in \mathbb{Z}^n \setminus \{\0\}} \int_{0}^{T} |\bj| |(\widetilde{\bb} \cdot \bj) \widehat{\bb} (t,\bj)| \, dt.
\end{align*}
The following proposition presents the corresponding estimate.
\begin{prop}\label{prop5.2}
Let $n \in \mathbb{N}$ with $n \geq 2$ and $m \in \mathbb{N}$ fulfilling $m > \max\{1 + 2r + \frac{n}{2},\,\, 4+\frac{n}{2}\}$. Suppose that $(\bv, \bb)$ is a smooth global solution to \eqref{equation}. Then there exists a constant $C > 0$ such that for all $T>0$,
\begin{align}\label{prop1.2}
 &\sum_{\bj \in \mathbb{Z}^n \setminus \{\0\}} \int_{0}^{T}| \widehat{\bb}(t,\bj)| \, dt+\sum_{\bj \in \mathbb{Z}^n \setminus \{\0\}} \int_{0}^{T} |\bj| |(\widetilde{\bb} \cdot \bj) \widehat{\bb} (t,\bj)| \, dt\notag\\[1ex]
\leq  & C\lVert(\bv_0, \bb_0) \rVert_{H^m}+C \sup_{t \in [0, T]} \| (\bv, \bb)(t) \|_{H^m} \left( \int_0^T \sum_\bj |\bj| |\widehat{\bv}(t,\bj)| \, dt + \int_0^T \sum_\bj |\widehat{\bb}(t,\bj)| \, dt \right).
\end{align}
\end{prop}
\begin{proof}
By virtue of $\eqref{2.4}_2$, we arrive at
\begin{align}
&\sum_{\bj \in \mathbb{Z}^n \setminus \{\0\}} \int_{0}^{T} |\widehat{\bb}(t,\bj)| \, dt\notag\\[1ex]
\leq& \sum_{\bj \in \mathbb{Z}^n \setminus \{\0\}} \int_{0}^{T} |\widehat{G}_1||\widehat{\boldsymbol{\psi}}_0|\, dt+\sum_{\bj \in \mathbb{Z}^n \setminus \{\0\}} \int_{0}^{T} \int_{0}^{t}
|\widehat{G}_1(t-\tau)||\widehat{\mathbf{N}}(\tau)| \, d\tau dt\notag\\[1ex]
&+\sum_{\bj \in \mathbb{Z}^n \setminus \{\0\}} \int_{0}^{T} |\widehat{G}_3||\widehat{\boldsymbol{\psi}}_0|  \, dt+\sum_{\bj \in \mathbb{Z}^n \setminus \{\0\}} \int_{0}^{T} \int_{0}^{t}
|\widehat{G}_3(t-\tau)||\widehat{\mathbf{N}}_2(\tau)| \, d\tau dt\notag\\[1ex]
=:&J_5 + J_6 + J_7 + J_8.\label{le''y}
\end{align}
Similarly,
\begin{align}
&\sum_{\bj \in \mathbb{Z}^n \setminus \{\0\}} \int_{0}^{T} |\bj| |(\widetilde{\bb} \cdot \bj) \widehat{\bb} (t,\bj)| \, dt\notag\\[1ex]
\leq& \sum_{\bj \in \mathbb{Z}^n \setminus \{\0\}} \int_{0}^{T} |\bj| |\widetilde{\bb} \cdot \bj||\widehat{G}_1||\widehat{\boldsymbol{\psi}}_0|\, dt+\sum_{\bj \in \mathbb{Z}^n \setminus \{\0\}} \int_{0}^{T} \int_{0}^{t}
|\bj| |\widetilde{\bb} \cdot \bj||\widehat{G}_1(t-\tau)||\widehat{\mathbf{N}}(\tau)| \, d\tau dt\notag\\[1ex]
&+\sum_{\bj \in \mathbb{Z}^n \setminus \{\0\}} \int_{0}^{T} |\bj| |\widetilde{\bb} \cdot \bj||\widehat{G}_3||\widehat{\boldsymbol{\psi}}_0|  \, dt+\sum_{\bj \in \mathbb{Z}^n \setminus \{\0\}} \int_{0}^{T} \int_{0}^{t}
|\bj| |\widetilde{\bb} \cdot \bj||\widehat{G}_3(t-\tau)||\widehat{\mathbf{N}}_2(\tau)| \, d\tau dt\notag\\[1ex]
=:&I_5 + I_6 + I_7 + I_8.\label{le''}
\end{align}
We begin with dealing with \(I_7\) and \(J_7\). Proposition~\ref{pro2} implies
\begin{align*}
J_7+I_7 &\leq \sum_{\bj \in \mathbb{Z}^n \setminus \{\0\}} \int_0^T e^{-\frac{t}{2}} |\widehat{\boldsymbol{\psi}}_0(\bj)| \, dt+C(|\widetilde{\bb}|) \sum_{\bj \in \mathbb{Z}^n \setminus \{\0\}} \int_0^T |\bj|^{2}e^{-\frac{t}{2}} |\widehat{\boldsymbol{\psi}}_0(\bj)| \, dt\\[1ex]
&\leq C \| \boldsymbol{\psi}_0 \|_{H^m},\quad {\rm for}\,\, m > 2 + \frac{n}{2}.
\end{align*}
For \( I_8 \) and \( J_8 \), leveraging  Proposition~\ref{pro2} and proceeding similarly as  \( J_2 \), we deduce for \( m > 2 + \tfrac{n}{2} \),
\begin{align*}
J_8+I_8&\leq C(|\widetilde{\bb}|)\int_0^T \sum_{\bj \in \mathbb{Z}^n \setminus \{\0\}}
|\bj|^{2}|(\widehat{\bb} * \widehat{\bv})(t,\bj)|  \, dt \\[1ex]
&\leq  C \int_0^T \left( \sum_{\bj \in \mathbb{Z}^n \setminus \{\0\}} |\bj|^{2}|\widehat{\bb}(t,\bj)| \right)
\left( \sum_{\bj \in \mathbb{Z}^n \setminus \{\0\}} |\widehat{\bv}(t,\bj)| \right) \,dt \\[1ex]
 &\quad+C \int_0^T \left( \sum_{\bj \in \mathbb{Z}^n \setminus \{\0\}} |\widehat{\bb}(t,\bj)| \right)
\left( \sum_{\bj \in \mathbb{Z}^n \setminus \{\0\}} |\bj|^{2} |\widehat{\bv}(t,\bj)| \right)\, dt\\[1ex]
&\leq C \sup_{t \in [0, T]} \| \bb(t) \|_{H^m} \int_0^T \sum_{\bj \in \mathbb{Z}^n \setminus \{\0\}}  |\widehat{\bv}(t,\bj)| \, dt \\[1ex]
 &\quad+C \sup_{t \in [0, T]} \| \bv(t) \|_{H^m} \int_0^T \sum_{\bj \in \mathbb{Z}^n \setminus \{\0\}}  |\widehat{\bb}(t,\bj)| \, dt
\\[1ex]
&\leq C \sup_{t \in [0, T]} \| (\bv, \bb)(t) \|_{H^m} \left( \int_0^T \sum_{\bj \in \mathbb{Z}^n \setminus \{\0\}} |\bj| |\hat{\bv}(t,\bj)| \, dt + \int_0^T \sum_{\bj \in \mathbb{Z}^n \setminus \{\0\}} |\widehat{\bb}(t,\bj)| \, dt \right).
\end{align*}
In what follows, we shall handle \(J_5\) and \(I_5\), and divide the analysis into three cases: \(\bj \in S_1\), \(\bj \in S_2\), and \(\bj \in S_3\):
\begin{align*}
J_5
&= \sum_{\bj \in S_1} \int_0^T |\widehat{G}_1(t,\bj)|\, |\hat{\boldsymbol{\psi}}_0(\bj)| \, dt
 + \sum_{\bj \in S_2} \int_0^T |\widehat{G}_1(t,\bj)|\, |\hat{\boldsymbol{\psi}}_0(\bj)| \, dt \\
&\quad + \sum_{\bj \in S_3} \int_0^T |\widehat{G}_1(t,\bj)|\, |\hat{\boldsymbol{\psi}}_0(\bj)| \, dt
=: J_{51} + J_{52} + J_{53},
\end{align*}
and
\begin{align*}
I_5
&= \sum_{\bj \in S_1} \int_0^T  |\bj|\, |\widetilde{\bb} \cdot \bj|\, |\widehat{G}_1(t,\bj)|\, |\hat{\boldsymbol{\psi}}_0(\bj)| \, dt
 + \sum_{\bj \in S_2} \int_0^T |\bj|\, |\widetilde{\bb} \cdot \bj|\, |\widehat{G}_1(t,\bj)|\, |\widehat{\boldsymbol{\psi}}_0(\bj)| \, dt \\
&\quad + \sum_{\bj \in S_3} \int_0^T  |\bj|\, |\widetilde{\bb} \cdot \bj|\, |\widehat{G}_1(t,\bj)|\, |\widehat{\boldsymbol{\psi}}_0(\bj)| \, dt
=: I_{51} + I_{52} + I_{53}.
\end{align*}
Through Proposition~\ref{pro2}, we infer
\begin{equation}\label{1.413}
|\widehat{G}_1(t,\bj)| \le
\begin{cases}
C|\bj|e^{-\frac{t}{4}}, & \bj \in S_1, \\[1ex]
Ce^{-\frac{t}{8}}, &  \bj \in S_2, \\[1ex]
Ce^{-|\widetilde{\bb} \cdot \bj|^2 t}, &  \bj \in S_3.
\end{cases}
\end{equation}
Hence,
\begin{align*}
I_{51}+I_{52}&\leq C(|\widetilde{\bb}|) \sum_{\bj \in S_1} \int_0^T |\bj|^3 e^{-\frac{t}{4} } |\widehat{\boldsymbol{\psi}}_0(\bj)| \, dt+C\sum_{\bj \in  S_2} \int_0^T |\bj|e^{-\frac{t}{8}} |\hat{\boldsymbol{\psi}}_0(\bj)| \, dt\\[1ex]
&\leq C \sum_{\bj \in S_1} |\bj|^3|\hat{\boldsymbol{\psi}}_0(\bj)|+ C \sum_{\bj \in S_2} |\bj||\widehat{\boldsymbol{\psi}}_0(\bj)|
\\[1ex]
&\leq C \| \boldsymbol{\psi}_0 \|_{H^m}, \quad {\rm for} \,\,m > 3 + \frac{n}{2}.
\end{align*}
Similarly,
\begin{align*}
J_{51}+J_{52}&\leq C \sum_{\bj \in S_1} \int_0^T |\bj| e^{-\frac{t}{4} } |\widehat{\boldsymbol{\psi}}_0(\bj)| \, dt+C \sum_{\bj \in  S_2} \int_0^T e^{-\frac{t}{8}} |\widehat{\boldsymbol{\psi}}_0(\bj)| \, dt\\[1ex]
&\leq C \sum_{\bj \in S_1} |\bj||\widehat{\boldsymbol{\psi}}_0(\bj)|+ C \sum_{\bj \in S_2} |\widehat{\boldsymbol{\psi}}_0(\bj)|
\\[1ex]
&\leq C \| \boldsymbol{\psi}_0 \|_{H^m}, \quad {\rm for} \,\,m > 1 + \frac{n}{2}.
\end{align*}
For \(I_{53}\), using the Diophantine condition and the fact
\begin{align}\label{2.3}
\int_{0}^{T} q e^{-q t} \, dt \le 1, \qquad \forall\, q \ge 0,\, T \ge 0,
\end{align}
we obtain, for \(m > 1 + r + \tfrac{n}{2}\),
\begin{align*}
I_{53}
&\le C \sum_{\bj \in S_3} \int_0^T |\bj|\, |\widetilde{\bb} \cdot \bj|\, e^{-(\widetilde{\bb} \cdot \bj)^2 t} |\widehat{\boldsymbol{\psi}}_0(\bj)| \, dt \\[0.5ex]
&\le C \sum_{\bj \in S_3} \frac{|\bj|}{|\widetilde{\bb} \cdot \bj|} \int_0^T |\widetilde{\bb} \cdot \bj|^2 e^{-|\widetilde{\bb} \cdot \bj|^2 t} \, dt \, |\widehat{\boldsymbol{\psi}}_0(\bj)| \\[0.5ex]
&\le C \sum_{\bj \in \mathbb{Z}^n \setminus \{\0\}} |\bj|^{1+r} |\widehat{\boldsymbol{\psi}}_0(\bj)|
\le C \| \boldsymbol{\psi}_0 \|_{H^m}.
\end{align*}
For \(J_{53}\), applying \eqref{2.3} yields
\begin{align*}
J_{53}&\leq \sum_{\bj \in S_3} \int_0^T \frac{1}{|\widetilde{\bb}\cdot \bj|^2}(\widetilde{\bb}\cdot \bj)^2 e^{-(\widetilde{\bb}\cdot \bj)^2t}| |\widehat{\boldsymbol{\psi}}_0(\bj)| \, dt \\[1ex]
&\leq  C \sum_{\bj \in S_3} \frac{1}{|\widetilde{\bb} \cdot \bj|^2} \left( \int_0^T {|\widetilde{\bb} \cdot \bj|^2} e^{-{|\widetilde{\bb} \cdot \bj|^2} t}\, dt \right) |\widehat{\boldsymbol{\psi}}_0(\bj)| \\[1ex]
&\leq C \sum_{\bj \in \mathbb{Z}^n \setminus \{\0\}} |\bj|^{2r} |\widehat{\boldsymbol{\psi}}_0(\bj)| \leq C \| \boldsymbol{\psi}_0 \|_{H^m}, \qquad \text{for } m > 2r + \tfrac{n}{2}.
\end{align*}
Finally, to estimate \(I_6\) and \(J_6\), we write
\begin{align*}
J_6
=&\sum_{\bj \in S_1} \int_0^T \int_0^t|\widehat{G}_1(t-\tau,\bj)|\, |\widehat{\mathbf{N}}(\tau,\bj)| \, d\tau dt\\[1ex]
&+\sum_{\bj \in S_2}\int_0^T \int_0^t|\widehat{G}_1(t-\tau,\bj)|\, |\widehat{\mathbf{N}}(\tau,\bj)| \, d\tau dt\\[1ex]
&+\sum_{\bj \in S_3}\int_0^T \int_0^t|\widehat{G}_1(t-\tau,\bj)|\, |\widehat{\mathbf{N}}(\tau,\bj)| \, d\tau dt\\[1ex]
=:&\, J_{61}+J_{62}+J_{63},
\end{align*}
and
\begin{align*}
I_6=&\sum_{\bj \in S_1} \int_0^T \int_0^t |\bj| |\widetilde{\bb} \cdot \bj||\widehat{G}_1(t-\tau,\bj)|\, |\widehat{\mathbf{N}}(\tau,\bj)| \, d\tau dt\\[1ex]
&+\sum_{\bj \in S_2}\int_0^T \int_0^t |\bj| |\widetilde{\bb} \cdot \bj||\widehat{G}_1(t-\tau,\bj)|\, |\widehat{\mathbf{N}}(\tau,\bj)| \, d\tau dt\\[1ex]
&+ \sum_{\bj \in S_3}\int_0^T \int_0^t |\bj| |\widetilde{\bb} \cdot \bj||\widehat{G}_1(t-\tau,\bj)|\, |\widehat{\mathbf{N}}(\tau,\bj)| \, d\tau dt\\[1ex]
=:&I_{61}+I_{62}+I_{63}.
\end{align*}
For \(I_{61}+I_{62}\), when \(m > 4 + \frac{n}{2}\), combining  \eqref{1.413}, \eqref{1.411} and \eqref{1.412},
by Young, Poincar\'e, and H\"older inequalities, we obtain
\begin{align*}
I_{61}+I_{62}
&\leq C \sum_{\bj \in S_1\cup S_2} \int_0^T |\bj|^3 |\widehat{\mathbf{N}}(t,\bj)| \, dt\\[1ex]
&\leq C \int_0^T\sum_{\bj \in \mathbb{Z}^n \setminus \{\0\}} |\bj|^4
\big(|\widehat{\bv}*\widehat{\bv}| + |\widehat{\bb}*\widehat{\bb}| + |\widehat{\bv}*\widehat{\bb}|\big)\, dt\\[1ex]
&\leq C \sup_{t \in [0, T]} \|(\bv,\bb)(t)\|_{H^m}
\left(\int_0^T\sum_{\bj \in \mathbb{Z}^n \setminus \{\0\}} |\bj|\,|\widehat{\bv}(t,\bj)|\,dt
+ \int_0^T\sum_{\bj \in \mathbb{Z}^n \setminus \{\0\}} |\widehat{\bb}(t,\bj)|\,dt\right).
\end{align*}
Similarly, for \(J_{61}+J_{62}\), one gets
\begin{align*}
J_{61}+J_{62}
&\leq C(|\widetilde{\bb}|) \sum_{\bj \in S_1\cup S_2} \int_0^T |\bj|\,|\widehat{\mathbf{N}}(t,\bj)|\, dt\\[1ex]
&\leq C \int_0^T\sum_{\bj \in \mathbb{Z}^n \setminus \{0\}} |\bj|^2
\big(|\widehat{\bv}*\widehat{\bv}| + |\widehat{\bb}*\widehat{\bb}| + |\widehat{\bv}*\widehat{\bb}|\big)\, dt\\[1ex]
&\leq C \sup_{t \in [0, T]} \|(\bv,\bb)(t)\|_{H^m}
\left(\int_0^T\sum_{\bj \in \mathbb{Z}^n \setminus \{\0\}} |\bj|\,|\widehat{\bv}(t,\bj)|\,dt
+ \int_0^T\sum_{\bj \in \mathbb{Z}^n \setminus \{\0\}} |\widehat{\bb}(t,\bj)|\,dt\right),
\end{align*}
for \(m > 2+\frac{n}{2}\). Here we used the Poincar\'e inequality in the last step as \(\bj \ne \0\).

For the high-frequency region \(S_3\), when \(m > 2+r+\frac{n}{2}\), we obtain
\begin{align*}
I_{63}\leq& C \sum_{\bj \in S_3} \int_0^T |\bj|^{1+r} |\widehat{\mathbf{N}}(t,\bj)| \, dt\\[1ex]
\leq& C \int_0^T\sum_{\bj \in S_3} |\bj|^{2+r}\left( |\widehat{\bv} * \widehat{\bv}| + |\widehat{\bb} * \widehat{\bb}| + |\widehat{\bv} * \widehat{\bb}| \right)\, dt
\\[1ex]
\leq& C \sup_{t \in [0, T]} \| (\bv, \bb)(t) \|_{H^m} \left( \int_0^T \sum_{\bj \in \mathbb{Z}^n \setminus \{\0\}} |\bj| |\widehat{\bv}(t,\bj)| \, dt + \int_0^T \sum_{\bj \in \mathbb{Z}^n \setminus \{\0\}} |\widehat{\bb}(t,\bj)| \, dt \right).
\end{align*}
Similarly, for \(m > 1+2r+\frac{n}{2}\), one has
\begin{align*}
J_{63}\leq& C \sum_{\bj \in S_3} \int_0^T |\bj|^{2r} |\widehat{\mathbf{N}}(t,\bj)| \, dt\\[1ex]
\leq& C \int_0^T\sum_{\bj \in S_3} |\bj|^{2r+1}\left( |\widehat{\bv} * \widehat{\bv}| + |\widehat{\bb} * \widehat{\bb}| + |\widehat{\bv} * \widehat{\bb}| \right)\, dt
\\[1ex]
\leq& C \sup_{t \in [0, T]} \| (\bv, \bb)(t) \|_{H^m} \left( \int_0^T \sum_{\bj \in \mathbb{Z}^n \setminus \{\0\}} |\bj| |\widehat{\bv}(t,\bj)| \, dt + \int_0^T \sum_{\bj \in \mathbb{Z}^n \setminus \{\0\}} |\widehat{\bb}(t,\bj)| \, dt \right).
\end{align*}
Combining the above estimates for \(I_5\)-\(I_8\) in \eqref{le''} and for \(J_5\)-\(J_8\) in \eqref{le''y} leads to the result in Proposition~\ref{prop5.2}.
\end{proof}

We are now in a position to prove the global existence part of Theorem~\ref{thm}.
Recall \eqref{12.24} as
\begin{equation}\label{12.24'}
E_m^2(T) \leq C(|\widetilde{\bb}|)\|(\bv_0,\bb_0)\|_{H^m}^2
+ C E_m^4(T)
+ C E_m^2(T) \int_0^T \|\nabla \bv(t)\|_{L^\infty}\,dt.
\end{equation}
To extend the solution globally, it suffices to control the last term on the right-hand side of \eqref{12.24'}, i.e.,
\[
\int_0^T \|\nabla \bv(t)\|_{L^\infty}\,dt.
\]
By applying  the Fourier transformation, we get for all \(t \ge 0\) that
\[
\|\nabla \bv(t)\|_{L^\infty}
\leq C \sum_{\bj\in \mathbb{Z}^n\setminus\{\0\}} |\bj|\,|\widehat{\bv}(t,\bj)|,
\qquad
\| \bb(t)\|_{L^\infty}
\leq C \sum_{\bj\in \mathbb{Z}^n\setminus\{\0\}}|\widehat{\bb}(t,\bj)|.
\]
Hence,
\begin{align*}
\int_0^T \big(\|\nabla \bv(t)\|_{L^\infty} + \|\bb(t)\|_{L^\infty}\big)\,dt
\leq C
\sum_{\bj\in \mathbb{Z}^n\setminus\{\0\}} \int_0^T |\bj|\,|\widehat{\bv}(t,\bj)|\,dt
+ C\sum_{\bj\in \mathbb{Z}^n\setminus\{\0\}} \int_0^T |\widehat{\bb}(t,\bj)|\,dt.
\end{align*}
Propositions~\ref{prop5.1} and \ref{prop5.2} give that
\begin{align}\label{prop-appliedy}
 &\sum_{\bj \in \mathbb{Z}^n \setminus \{0\}} \int_{0}^{T} |\bj|\left| \widehat{\bv}(t,\bj) \right| \, dt
+ \sum_{\bj \in \mathbb{Z}^n \setminus \{0\}} \int_{0}^{T}| \widehat{\bb}(t,\bj)| \, dt\leq C\|(\bv_0,\bb_0)\|_{H^m} \notag\\
&\quad
+ C_1 \sup_{t \in [0,T]} \|(\bv, \bb)(t)\|_{H^m}
\left( \sum_{\bj \in \mathbb{Z}^n \setminus \{0\}}\int_0^T |\bj|\,|\widehat \bv(t,\bj)| \, dt
+ \sum_{\bj \in \mathbb{Z}^n \setminus \{0\}}\int_0^T|\widehat \bb(t,\bj)| \, dt \right),
\end{align}
for some constant \(C_1>0\).

Taking the initial data sufficiently small as in \eqref{smallcondition}, the last term on the right-hand side of \eqref{prop-appliedy} could be absorbed into the left-hand side by a standard bootstrap argument. Then,
\begin{align}\label{Linftyq}
\int_0^T \|\nabla \bv(t)\|_{L^\infty}\,dt
\leq \int_0^T \big(\|\nabla \bv(t)\|_{L^\infty} + \|\bb(t)\|_{L^\infty}\big)\,dt
\leq C\|(\bv_0,\bb_0)\|_{H^m}, \quad \forall\, T>0.
\end{align}
Substituting \eqref{Linftyq} into \eqref{12.24'} and resorting to the smallness assumption \eqref{smallcondition} once again, we know that
\begin{equation}\label{12.24'y}
E_m^2(T) \leq C(|\widetilde{\bb}|)\|(\bv_0,\bb_0)\|_{H^m}^2 + C_2 E_m^4(T),
\end{equation}
for some constant \(C_2>0\).

To close the bootstrap, we assume
\begin{align}\label{(9)}
E_m^2(T) \leq \frac{1}{2C_2}.
\end{align}
Then \eqref{12.24'y} implies
\[
E_m^2(T) \leq C(|\widetilde{\bb}|)E_m^2(0),
\]
where $E_m(0):=\|(\bv_0,\bb_0)\|_{H^m}.$
In addition, if the initial data are sufficiently small such that
\[
E_m^2(0) \leq \frac{1}{4C_2},
\]
then \(E_m^2(T) < \frac{1}{4C_2}\), which is consistent with \eqref{(9)} and therefore ends the bootstrap argument.
This gives rise to
\begin{align}\label{eq:1.2qw}
\sup_{t \in [0, \infty)} \lVert (\bv, \bb)(t) \rVert_{H^m}^2
+ \int_0^\infty \lVert \bv(t) \rVert_{H^m}^2 \,{d}t
+ \int_0^\infty \lVert \bb(t) \rVert_{H^{m-1-r}}^2 \, {d}t
\leq C \lVert (\bv_0, \bb_0) \rVert_{H^m}^2,
\end{align}
which implies the global-in-time existence of smooth solutions and the uniform bound \eqref{eq:1.2} in Theorem \ref{thm}.

\subsection{Temporal decay estimate}\label{4/3}
In this subsection, we verify  the temporal decay estimate presented  in Theorem~\ref{thm}. Let $(\bv, \bb)$ be a smooth global-in-time solution to \eqref{equation}. The proof depends  on a time-weighted energy method applied to the modified energy functional \(Q_s(t)\) defined in \eqref{lya}.

Following the argument in \eqref{2.33}, for any $s \in [0, m]$ with \( m > 2 + r + \frac{n}{2} \), one gets
\begin{align}\label{2.333t}
\frac{1}{2}\frac{d}{dt} \lVert (\bv, \bb) \rVert_{H^s}^2 +\lVert \bv \rVert_{H^s}^2
\leq C \left( \lVert \nabla \bv \rVert_{L^{\infty}} + \lVert \bb \rVert_{H^{m-1-r}}^2 \right) \lVert (\bv, \bb) \rVert_{H^s}^2.
\end{align}
Combining \eqref{2.333t} with \eqref{2.333}, we have
\begin{align}\label{12.2ab}
\frac{1}{2} \frac{d}{dt} Q_s(t)
\leq\; & -a \left\lVert \bv \right\rVert_{H^s}^2 - \frac{1}{4} \left\lVert \Lambda^{-1} (\widetilde{\bb} \cdot \nabla \bb) \right\rVert_{H^s}^2
+ \left( \frac{1}{2} + \frac{|\widetilde{\bb}|^2}{2} \right) \left\lVert \bv \right\rVert_{H^s}^2 \notag\\
& + C \lVert (\bv, \bb) \rVert_{H^s}^2 \left( \left\lVert \nabla \bv \right\rVert_{L^{\infty}} + \lVert \bb \rVert_{H^{m-1-r}}^2 \right) \notag\\
\leq\; & -\frac{1}{4} \left( \left\lVert \bv \right\rVert_{H^s}^2 + \left\lVert \Lambda^{-1} (\widetilde{\bb} \cdot \nabla \bb) \right\rVert_{H^s}^2 \right)
+ C \lVert (\bv, \bb) \rVert_{H^s}^2 \left( \left\lVert \nabla \bv \right\rVert_{L^{\infty}} + \lVert \bb \rVert_{H^{m-1-r}}^2 \right).
\end{align}
Making use of Lemmas~\ref{lem:zero-mean-fractional} and \ref{2.1}, together with Poincar\'e inequality, one infers
\begin{align*}
\left\lVert \Lambda^{-1-r} \bb \right\rVert_{\dot{H}^s}
&\leq \left\lVert \Lambda^{-1-r} \bb \right\rVert_{H^s}
\leq c \left\lVert \widetilde{\bb} \cdot \nabla(\Lambda^{-1-r} \bb) \right\rVert_{H^{s+r}} \\
&= c \left\lVert \Lambda^{-1-r}(\widetilde{\bb} \cdot \nabla \bb) \right\rVert_{H^{s+r}}
\leq c_1 \left\lVert \Lambda^{-1}(\widetilde{\bb} \cdot \nabla \bb) \right\rVert_{H^s},
\end{align*}
where \( c \) is  the constant appearing in \eqref{Diophantine}. Then \eqref{12.2ab} could be rewritten as
\begin{align}\label{6.4}
\frac{d}{dt} Q_s(t)
\leq -\frac{c^*}{4} \left( \left\lVert \bv \right\rVert_{H^s}^2 + \left\lVert \Lambda^{-1-r} \bb \right\rVert_{\dot{H}^s}^2 \right)
+C \lVert (\bv, \bb) \rVert_{H^s}^2 \left( \left\lVert \nabla \bv \right\rVert_{L^{\infty}}+ \lVert \bb \rVert_{H^{m-1-r}}^2 \right),
\end{align}
where \( c^* := \min\{1, \frac{1}{c_1}\} \).

Let \( M > 0 \) be a constant to be specified later such that \( \frac{a}{M} \leq 1 \). It holds through Plancherel's theorem that
\begin{align*}
&\frac{a}{M}\lVert\Lambda^{s}\bb\rVert_{L^2}^2-\lVert\Lambda^{-1 - r}\bb\rVert_{\dot{H}^{s}}^2\\[1ex]
=&\sum_{|\bj|\neq0}\left(\frac{a}{M}|\bj|^{2s}-|\bj|^{2s-2-2r}\right)|\widehat{\bb}(\bj)|^2\\[1ex]
\leq&\frac{a}{M}\sum_{\frac{a}{M}>|\bj|^{-2 - 2r}}|\bj|^{2s}|\widehat{\bb}(\bj)|^2\\[1ex]
=&\left(\frac{a}{M}\right)^{\frac{m - s}{1 + r}}\sum_{\frac{a}{M}>|\bj|^{-2 - 2r}}\left(\frac{M}{a}\right)^{\frac{m -s}{1 + r}-1}|\bj|^{2s - 2m+2+2r}|\bj|^{2m-2-2r}|\widehat{\bb}(\bj)|^2\\[1ex]
\leq&\left(\frac{a}{M}\right)^{\frac{m -s}{1 + r}}\lVert
\Lambda^{m-1 - r}\bb\rVert_{L^2}^2\leq\frac{C}{M^{\frac{m - s}{1 + r}}}  \|\bb\|^2_{H^{m-1-r}}.
\end{align*}
Hence, the first term on the right-hand side of \eqref{6.4} becomes
\begin{align}\label{6.4aab}
&- \frac{c^{*}}{4} \left( \left\lVert \bv \right\rVert_{H^s}^2 + \left\lVert \Lambda^{-1-r} \bb \right\rVert_{\dot{H}^s}^2 \right)\notag\\
= &- \frac{c^{*}}{4} \left\lVert \bv \right\rVert_{H^s}^2 - \frac{c^* a}{4M} \left\lVert \Lambda^s \bb \right\rVert_{L^2}^2
+ \frac{c^*}{4} \left( \frac{a}{M} \left\lVert \Lambda^s \bb \right\rVert_{L^2}^2 - \left\lVert \Lambda^{-1 - r} \bb \right\rVert_{\dot{H}^s}^2 \right) \notag \\
\leq& - \frac{c^* a}{4 M c_0} \left( \left\lVert \bv \right\rVert_{H^s}^2 + \left\lVert \bb \right\rVert_{H^s}^2 \right)
+\frac{C}{M^{\frac{m - s}{1 + r}}}\|\bb\|^2_{H^{m-1-r}} \notag \\
\leq &- \frac{c^*}{8 M c_0}Q_s(t)
+ \frac{C}{M^{\frac{m - s}{1 + r}}}\|\bb\|^2_{H^{m-1-r}},
\end{align}
where we used \eqref{12.21}, and the fact \( \left\lVert \bb \right\rVert_{H^s} \leq c_0 \left\lVert \Lambda^s \bb \right\rVert_{L^2} \) which follows from the Poincar\'e inequality thanks to  the mean-zero condition \eqref{meanzero2}. Substituting \eqref{6.4aab} into \eqref{6.4} yields
\begin{align}\label{6.6}
\frac{d}{dt} Q_s(t)
\leq - \frac{c^*}{8 M c_0} Q_s(t)
+ \frac{C}{M^{\frac{m - s}{1 + r}}} \left\lVert \bb \right\rVert_{H^{m - 1 - r}}^2
+ C \left\lVert (\bv, \bb) \right\rVert_{H^s}^2\left( \left\lVert \nabla \bv \right\rVert_{L^{\infty}}+ \lVert \bb \rVert_{H^{m-1-r}}^2 \right).
\end{align}
 Taking \( M = a + \frac{c^* t}{8 c_0 \frac{m - s}{1 + r}} \) and multiplying  both sides of \eqref{6.6} with \( M^{\frac{m - s}{1 + r}} \), we infer
\begin{align*}
\frac{d}{dt} \left( M^{\frac{m - s}{1 + r}} Q_s(t) \right)
\leq C \left\lVert \bb \right\rVert_{H^{m - 1 - r}}^2
+ C\left( M^{\frac{m - s}{1 + r}} Q_s(t) \right)\left( \left\lVert \nabla \bv \right\rVert_{L^{\infty}}+ \lVert \bb \rVert_{H^{m-1-r}}^2 \right).
\end{align*}
Then Gr\"onwall's inequality and \eqref{eq:1.2qw} give the decay estimate \eqref{finaldecay}. So far, we complete the proof of Theorem~\ref{thm}.

\end{document}